\documentclass[10pt]{amsart}

\RequirePackage{etex}
\usepackage[utf8]{inputenc}
\usepackage[T1]{fontenc}
\usepackage[english]{babel}

\usepackage[mathscr]{euscript}
\usepackage{caption}
\usepackage{subcaption}
\usepackage{datetime}
\usepackage{url}
\usepackage{verbatim}
\usepackage{float}
\usepackage{lipsum}
\usepackage[dvipsnames]{xcolor}

\usepackage{titletoc}
\usepackage{textcomp}
\usepackage{anysize}
\usepackage{fancyhdr}
\usepackage{calrsfs}

\usepackage{amssymb}
\usepackage{amsmath,mathtools}
\usepackage{amsfonts}
\usepackage{amscd}
\usepackage{amsthm}
\usepackage{dsfont}

\usepackage{latexsym}
\usepackage{enumerate}
\usepackage[shortlabels]{enumitem}
\usepackage{array}
\usepackage{stackrel}
\usepackage{stmaryrd}
\usepackage{stackengine}

\usepackage{graphicx}
	\graphicspath{ {Images/} }
\usepackage{color,graphics}
\usepackage{tikz}
	\usetikzlibrary{arrows,automata,patterns,decorations.pathmorphing,decorations.pathreplacing, decorations.markings,shapes.misc}
\usepackage{tikz-cd}
	\usetikzlibrary{matrix,arrows,calc,automata,trees}
\usepackage[all,cmtip]{xy}
\usepackage{tkz-graph}
\usepackage{pgfplots,caption}
\pgfplotsset{compat=1.9}

\usepackage{lscape}
\usepackage[active]{srcltx}
\usepackage{rotating}

\usepackage[hypertexnames=false]{hyperref}

\AtBeginDocument{
   \def\MR#1{}
}

\makeatletter
\renewcommand{\tocsection}[3]{
  \indentlabel{\@ifnotempty{#2}{\bfseries\ignorespaces#1 #2\quad}}\bfseries#3}
\renewcommand{\tocsubsection}[3]{
  \indentlabel{\@ifnotempty{#2}{\ignorespaces#1 #2\quad}}#3}
\renewcommand{\tocsubsubsection}[3]{
  \indentlabel{\@ifnotempty{#2}{\ignorespaces#1 #2\quad}}#3}

\newcommand\@dotsep{4.5}
\def\@tocline#1#2#3#4#5#6#7{\relax
  \ifnum #1>\c@tocdepth 
  \else
    \par \addpenalty\@secpenalty\addvspace{#2}%
    \begingroup \hyphenpenalty\@M
    \@ifempty{#4}{%
      \@tempdima\csname r@tocindent\number#1\endcsname\relax
    }{%
      \@tempdima#4\relax
    }%
    \parindent\z@ \leftskip#3\relax \advance\leftskip\@tempdima\relax
    \rightskip\@pnumwidth plus1em \parfillskip-\@pnumwidth
    #5\leavevmode\hskip-\@tempdima{#6}\nobreak
    \leaders\hbox{$\m@th\mkern \@dotsep mu\hbox{.}\mkern \@dotsep mu$}\hfill
    \nobreak
    \hbox to\@pnumwidth{\@tocpagenum{\ifnum#1=1\bfseries\fi#7}}\par
    \nobreak
    \endgroup
  \fi}
\AtBeginDocument{%
\expandafter\renewcommand\csname r@tocindent0\endcsname{0pt}
}
\def\l@subsection{\@tocline{2}{0pt}{2.5pc}{5pc}{}}

\makeatother

\fancypagestyle{my_own_style}{
		\fancyhf{}
		\fancyfoot[C]{Page \thepage}
}

\newcommand{\N}{{\mathbb N}}
\newcommand{\Z}{{\mathbb Z}}

\newcommand{\R}{{\mathbb R}}

\DeclareMathAlphabet{\pazocal}{OMS}{zplm}{m}{n}

\newcommand{\calB}{{\pazocal B}}

\newcommand{\calO}{{\pazocal O}}
\newcommand{\calP}{{\pazocal P}}

\newcommand{\ra}{\rightarrow}

\newcommand{\ol}{\overline}
\newcommand{\ul}{\underline}

\newcommand{\xddots}{%
  \raise 4pt \hbox {.}
  \mkern 6mu
  \raise 1pt \hbox {.}
  \mkern 6mu
  \raise -2pt \hbox {.}
}

\numberwithin{equation}{section}

\theoremstyle{plain}

\newtheorem{theorem}{Theorem}[section]
\newtheorem*{theorem*}{Theorem}
\newtheorem{lemma}[theorem]{Lemma}
\newtheorem{proposition}[theorem]{Proposition}
\newtheorem{corollary}[theorem]{Corollary}

\theoremstyle{definition}

\newtheorem{definition}[theorem]{Definition}
\newtheorem{examples}[theorem]{Examples}

\newtheorem*{remark*}{Remark}

\newtheorem*{remarks*}{Remarks}

\newtheorem*{assumption*}{Assumption}

\newtheorem{observations}[theorem]{Observations}
\newtheorem{construction}[theorem]{Construction}

\newtheorem{notation}[theorem]{Notation}

\makeatletter
\newenvironment{example}[1][]{%
  \refstepcounter{theorem}%
  \par\addvspace{\topsep}%
  \noindent
  \textbf{Example \thetheorem.}
  \if\relax\detokenize{#1}\relax\else\ \textbf{#1.}\fi
}{%
  \par\addvspace{\topsep}%
}
\makeatother

\tikzset{
  on each segment/.style={
    decorate,
    decoration={
      show path construction,
      moveto code={},
      lineto code={
        \path [#1]
        (\tikzinputsegmentfirst) -- (\tikzinputsegmentlast);
      },
      curveto code={
        \path [#1] (\tikzinputsegmentfirst)
        .. controls
        (\tikzinputsegmentsupporta) and (\tikzinputsegmentsupportb)
        ..
        (\tikzinputsegmentlast);
      },
      closepath code={
        \path [#1]
        (\tikzinputsegmentfirst) -- (\tikzinputsegmentlast);
      },
    },
  },
  mid arrow/.style={postaction={decorate,decoration={
        markings,
        mark=at position .5 with {\arrow[#1]{stealth}}
      }}},
}

\tikzset{
	cross/.style={cross out, draw=black, minimum size=2*(#1-\pgflinewidth), inner sep=0pt, outer sep=0pt},
	cross/.default={1pt}
}

\tikzset{
  symbol/.style={
    draw=none,
    every to/.append style={
      edge node={node [sloped, allow upside down, auto=false]{$#1$}}}
  }
}

\title[Combinatorial approximations of dynamical systems: a separated graph approach]{Combinatorial approximations of dynamical systems: a separated graph approach}
\author{Joan Claramunt}
\address[J. Claramunt]{Av. Diagonal, 647, Planta 3, Pavelló H, Les Corts, 08028 Barcelona.}
\email{joan.claramunt@upc.edu}


\thanks{This work was supported by grants PID2023-147110NB-I00 and PID2023-146758NB-I00 funded by MICIU/AEI/10.13039/501100011033 and by ERDF/EU}

\date{\today}

\begin{document}

\pagestyle{plain}
 
\begin{abstract}
%
%
Separated graphs provide a powerful combinatorial tool for approximating dynamical systems. This paper details the explicit construction of Bratteli-like separated graphs --- a generalization of classical Bratteli diagrams --- that encode the dynamics of a homeomorphism $h$ on a totally disconnected, compact metric space $X$. Unlike standard approaches, the separated graph framework allows us to explicitly disentangle the static structure of the space from the dynamics of the homeomorphism. We provide a step-by-step exposition of this construction applied to four fundamental examples: the two-sided shift, the bit-wise NOT (global flip) map, the classical odometer map and the shift map on the one-point compactification of the integers. Finally, we briefly discuss how minimal (and, more generally, essentially minimal) dynamical systems can be read directly from the separated graph. This approach builds upon recent work by P. Ara and the author, which provides a graph-theoretic model for dynamical systems given by surjective local homeomorphisms defined on totally disconnected compact metric spaces.
\end{abstract}

\maketitle

\tableofcontents

\normalsize

\section{Introduction}\label{section-introduction}

The study of topological dynamical systems on totally disconnected compact metric spaces -- often referred to as Cantor dynamics, cf. \cite[Chapter 19]{GKPT2018} -- occupies a central role in modern dynamical systems. These spaces arise naturally in a variety of contexts. Most notably, they appear as shift spaces in symbolic dynamics \cite{Kitchens1998, LM2021, KL2016, GKPT2018, Bruin2022}, as boundaries of discrete groups and rooted trees in geometric group theory and arboreal dynamics \cite{Nekrashevych2005, HLvL2021}, and as maximal invariant sets (such as the Smale horseshoe) in hyperbolic smooth dynamical systems \cite{KH1995, Kurka2003}. A defining feature of totally disconnected compact metric spaces is that their topology is generated by a basis of clopen sets. This property is highly advantageous dynamically: it allows one to study the global action of a homeomorphism $h \colon X \to X$ through a sequence of finite combinatorial approximations. The premier tool for this is the use of Kakutani-Rohlin partitions. Originally developed in ergodic theory \cite{Kakutani1943, Rohlin1948} (cf. \cite{Glasner2003}), these partitions were adapted to topological dynamics by Putnam \cite{Putnam1989_article} and Herman, Putnam, and Skau \cite{HPS1992}, and have since become a ubiquitous method for modeling both minimal and aperiodic Cantor systems \cite{Medynets2006, DK2019}.

The modern era of Cantor dynamics is heavily characterized by the use of combinatorial, algebraic, and geometric invariants to classify systems up to various notions of equivalence. A cornerstone of this approach is the theory of Bratteli-Vershik models. The underlying combinatorial structures, Bratteli diagrams, were originally introduced in 1972 by O. Bratteli to classify approximately finite-dimensional (AF) $C^*$-algebras \cite{Bratteli1972}. In the 1980s, A. Vershik pioneered the use of these diagrams in ergodic theory, demonstrating that every ergodic measure-preserving transformation is isomorphic to an adic transformation on the path space of a Bratteli diagram \cite{Vershik1981}.

The pivotal bridge to topological dynamics was established by Herman, Putnam, and Skau \cite{HPS1992}, who adapted this framework to Cantor spaces. They proved that every minimal homeomorphism of a Cantor set can be represented as a Vershik map acting on the space of infinite paths of an ordered Bratteli diagram. This dictionary between topology and combinatorics ignited a vast research program. For instance, Giordano, Putnam, and Skau utilized Bratteli diagrams and ordered $K$-theory to completely classify Cantor minimal systems up to strong orbit equivalence and orbit equivalence \cite{GPS1995}.

Following \cite{HPS1992}, a significant direction of research has been the extension of Bratteli-Vershik models beyond minimal systems. Medynets \cite{Medynets2006} generalized the construction to aperiodic homeomorphisms of Cantor spaces. More recently, Downarowicz and Karpel \cite{DK2019} provided a topological characterization of which zero-dimensional systems are ``Bratteli-Vershikizable'', showing that such models exist if and only if the set of aperiodic points is dense, or its closure misses exactly one periodic orbit. For arbitrary compact metric zero-dimensional dynamical systems, Shimomura \cite{Shimomura2020} and others have further generalized these combinatorial models. These techniques have also been adapted to non-compact spaces, such as in standard Borel dynamics \cite{BDK2006}. Furthermore, these combinatorial models play an important role in the study of topological full groups \cite{Matui2006}. This connection has led to significant applications in group theory; notably, Juschenko and Monod utilized the topological full groups of Cantor minimal systems to provide the first examples of finitely generated, infinite, simple, amenable groups \cite{Matui2006, JM2013}.

Building upon this tradition of combinatorial encoding, recent work by P. Ara and the author \cite{AC2024} introduced a novel approximation method for dynamical systems $(X,h)$ using \textit{separated graphs}. This framework, originally developed in \cite{AG2011, AG2012} to handle $C^*$-algebras that escape the reach of standard graph theory, has since been used to address several problems in the field. Notably, these graphs were used to investigate the Realization Problem for refinement monoids \cite{ABP2020} and were instrumental in constructing counterexamples to the Tarski alternative for group actions \cite{AE2014}.

In our approach, the approximation is achieved by choosing an $h$-refined sequence of partitions of $X$ (see Definition \ref{definition:h_refined}) to associate an \textit{$h$-diagram} $(F,D)$ to the system. It was shown that this association is well-behaved at the level of $C^*$-algebras, naturally relating the separated graph $C^*$-algebra of $(F,D)$ to the $C^*$-algebra of the transformation groupoid of the system. Furthermore, this construction gave rise to a new class of dynamical systems, termed generalized finite shifts, which subsume one- and two-sided shifts of finite type.

This article, which expands upon a lecture given by the author at the \textit{School XXIV School of Mathematics <<Lluis Santaló>> -- Structure and Approximation of $C^*$-algebras} and the previously mentioned paper \cite{AC2024}, serves a dual purpose. First, in the spirit of the original lecture, our exposition emphasizes the use of concrete examples. To illustrate the different levels of abstraction, we systematically apply the general construction from \cite{AC2024} to four fundamental dynamical systems. These examples are carefully chosen to exhibit a variety of dynamical behaviors, thereby clarifying the mechanics of the $h$-diagram and building the reader's intuition.

Second, we build upon this expository foundation to present new results demonstrating how specific dynamical properties can be read directly from the combinatorial structure of $(F,D)$. In particular, we focus on two extremes of dynamical behavior: systems that are globally periodic, and systems that are essentially minimal. We introduce purely combinatorial conditions on $(F,D)$ that capture these behaviors, obtaining the following results (see Theorems \ref{theorem:global_periodicity} and \ref{theorem:ess_min}, also Definitions \ref{definition:global_periodicity} and \ref{definition:property_em}, respectively):
\begin{theorem*}
Let $X$ be a totally disconnected compact metric space and $h$ a homeomorphism on $X$. Let $(F,D)$ denote an $h$-diagram associated with $(X,h)$.
\begin{enumerate}
\item[(1)] $(X,h)$ is an essentially minimal dynamical system if and only if $(F,D)$ is an essentially minimal $h$-diagram.
\item[(2)] For any $m \in \N$, it holds that $h^m = \mathrm{id}_X$ if and only if $(F,D)$ has global periodicity $m$.
\end{enumerate}
\end{theorem*}
\noindent As a consequence, we also obtain a combinatorial characterization of minimality (Corollary \ref{corollary-minimality}).

The paper is organized as follows. Section \ref{section-setup} sets out the necessary preliminaries concerning Cantor dynamics, sequences of refined partitions, and separated graphs. In Section \ref{section-the.construction}, we apply the general construction to our selected fundamental examples, providing explicit derivations of their associated $h$-diagrams. Finally, Section \ref{section:dynamical_properties} introduces the combinatorial notions of global periodicity and essential minimality for $h$-diagrams, relating them to their counterparts in the underlying dynamical system.

\section{The setup}\label{section-setup}

\subsection{Dynamical systems and partitions}\label{subsection-dynamical.systems.partitions}

We largely follow the framework of \cite{AC2024}, reproducing here some of the key concepts and proofs for the sake of completeness and to ensure the exposition is self-contained. We warn the reader that some notation and conventions may differ slightly from those in \cite{AC2024}.

Our dynamical systems of interest are pairs $(X,h)$ consisting of a homeomorphism $h$ defined on a totally disconnected compact metric space $X$. In most of our examples, $X$ will be taken to be the Cantor set, though the construction applies generally. In this paper, a \textit{partition} of the space $X$ is a finite collection $\calP$ of non-empty, pairwise disjoint closed and open (clopen) subsets of $X$ --- hereafter referred to as \textit{block} --- whose union is the whole space $X$.

Partitions of $X$ serve as our primary approximation tools for the space. To this end, we should be able to compare two partitions and decide which provides a finer approximation of $X$. This is achieved by introducing a partial order on the set of partitions of $X$: a partition $\calP_2$ is said to be \textit{finer} than another partition $\calP_1$ if every block of $\calP_2$ is contained in a (necessarily unique) block of $\calP_1$. This defines a partial order $\precsim$, where we write $\calP_2 \precsim \calP_1$ to denote that $\calP_2$ is finer than $\calP_1$.

Because $X$ is compact and metrizable, we can always find partitions of $X$ that are arbitrarily fine. Let $d$ denote a metric on $X$ witnessing the metrizability of its topology. The \textit{diameter} of a block $Z \in \calP$ is defined as the supremum of the distances between any two points in $Z$:
$$\operatorname{diam}(Z) := \sup\{d(x,y) \mid x,y \in Z\}.$$
We have the following straightforward approximation result.

\begin{lemma}\label{lemma:first_approx}
Given any $\varepsilon >0$, there exists a partition $\calP$ of $X$ such that $\operatorname{diam}(Z) < \varepsilon$ for every block $Z \in \calP$.
\end{lemma} 
Indeed, since $X$ admits a basis of clopen sets, for every $x \in X$ we can find a clopen neighborhood $U_x$ of $x$ with diameter less than $\varepsilon$. The family $\{U_x\}_{x \in X}$ is an open cover of $X$, so by compactness we can extract a finite subcover $\{U_{x_1}, \dots, U_{x_n}\}$. By taking successive set differences, we obtain from this finite cover a finite collection of pairwise disjoint clopen sets $\{Z_1, \dots, Z_m\}$ that still covers $X$. Since each $Z_i$ is contained in some $U_{x_j}$, we have $\text{diam}(Z_i) < \varepsilon$ for all $i \in \{1, \dots, m\}$.

This result is the first step toward constructing a combinatorial model for $X$ by means of sequences of finer and finer partitions of $X$. We start by establishing the existence of such sequences.

\begin{lemma}\label{lemma-technical.1}
There always exist \emph{refined sequences of partitions} of $X$, that is, sequences $\{\calP_n\}_{n \geq 0}$ of partitions of $X$ satisfying the following conditions:
\begin{enumerate}
\item[(a)] each $\calP_{n+1}$ is finer than $\calP_n$ for $n \geq 0$;
\item[(b)] the collection $\bigcup_{n \geq 0} \calP_n$ forms a basis for the topology of $X$;
\item[(c)] the maximum diameter $\max\{\operatorname{diam}(Z) \mid Z \in \calP_n\}$ tends to zero as $n \to \infty$.
\end{enumerate}
\end{lemma}
\begin{proof}
For each $n \geq 0$, due to Lemma \ref{lemma:first_approx} we can find a partition $\calP''_n$ such that $\max\{\text{diam}(Z) \mid Z \in \calP''_n\} < 1/2^n$. While the sequence $\{\calP''_n\}_{n \geq 0}$ satisfies condition (c), we have no guarantee that it satisfies the other conditions. For that, we first construct inductively a new sequence of partitions satisfying both (a) and (c).

Set $\calP'_0 := \calP''_0$, and suppose we have constructed $\calP'_0, \dots, \calP'_m$ for some $m \geq 0$. To construct $\calP'_{m+1}$, we consider the collection of all \textit{non-empty} intersections of the form $Z' \cap Z''$, where $Z \in \calP'_m$ and $Z' \in \calP''_{m+1}$. Since both $\calP'_m$ and $\calP''_{m+1}$ are partitions of $X$, these intersections are clopen, pairwise disjoint, and cover the whole space $X$. Therefore, $\calP'_{m+1}$ is a \textit{bona-fide} partition of $X$. By definition, $\calP'_{m+1}$ is finer than both $\calP'_m$ and $\calP''_{m+1}$, so conditions (a) and (c) are satisfied.

To ensure that condition (b) is also satisfied, we construct a final sequence of partitions $\{\calP_n\}_{n \geq 0}$ as follows. Since $X$ is totally disconnected and metrizable, it admits a countable basis of clopen sets $\calB = \{B_n\}_{n \geq 0}$. We define $\calP_0$ as the set of all \textit{non-empty} intersections of the form $Z' \cap B_0$ or $Z' \cap \left(X \setminus B_0\right)$, where $Z' \in \calP'_0$. Inductively, suppose we have constructed $\calP_0, \dots, \calP_m$ for some $m \geq 0$. To construct $\calP_{m+1}$, we consider the collection of all \textit{non-empty} intersections of the form $Z \cap Z' \cap B_{m+1}$ or $Z \cap Z' \cap \left(X \setminus B_{m+1}\right)$, where $Z \in \calP_m$ and $Z' \in \calP'_{m+1}$.

This construction yields the required sequence of partitions $\{\calP_n\}_{n \geq 0}$. Indeed, since each $\calP'_n$ is a partition of $X$, so is $\calP_n$. Furthermore, $\calP_{n+1}$ is finer than both $\calP_n$ and $\calP'_n$ by construction, ensuring conditions (a) and (c). Finally, given any open set $U \subseteq X$ and any point $x \in U$, we can find a a basis element $B_N \in \calB$ such that $x \in B_N \subseteq U$. Since $\calP_N$ is a partition of $X$ constructed by intersecting with $B_N$ and its complement, there exists a block $Z \in \calP_N$ such that $x \in Z \subseteq B_N \subseteq U$. This shows that $\bigcup_{n \geq 0} \calP_n$ forms a basis for the topology of $X$, verifying condition (b).
\end{proof}

The process used in the proof of Lemma \ref{lemma-technical.1} to construct a new partition that refines several given partitions is called \textit{wedging}, which we formalize as follows. Given a finite collection of partitions $\calP_1, \dots, \calP_r$ of $X$, we define its \textit{wedge}, denoted $\calP_1 \wedge \cdots \wedge \calP_r$, to be the partition consisting of all \textit{non-empty} intersections $Z_1 \cap \cdots \cap Z_r$, where $Z_i \in \calP_i$ for all $i \in \{1, \dots, r\}$. It is clear that $\calP_1 \wedge \cdots \wedge \calP_r$ is a partition of $X$ that is finer than all the original partitions $\calP_1, \dots, \calP_r$.\newline

Thus far, we have established the existence of sequences of partitions that approximate the space $X$ at any level of precision; we have inductively sliced $X$ into clopen sets of arbitrarily small diameter such that any open subset of $X$ can be written as a union of a collection of these slices. We now incorporate the homeomorphism $h$ into the picture by requiring our refined sequences of partitions to be dynamically compatible. Specifically, a given partition must refine not only the previous one, but also its translates under $h$ and $h^{-1}$.

\begin{definition}\label{definition:h_refined}
A refined sequence of partitions $\{\calP_n\}_{n \geq 0}$ of $X$ is called \textit{$h$-refined} if it further satisfies the condition
$$\calP_{n+1} \precsim h(\calP_n) \wedge h^{-1}(\calP_n)$$
for all $n \geq 0$.
\end{definition}

We remark that the definition of an $h$-refined sequence given here differs slightly from \cite[Definition 3.18]{AC2024}, where $\calP_{n+1}$ is required to refine either $h(\calP_n)$ or $h^{-1}(\calP_n)$ depending on the parity of $n$. However, this distinction is immaterial for our purposes. By \cite[Theorem 3.28]{AC2024}, the equivalence class of the $h$-diagram constructed from $(X,h)$ is independent of the specific choice of the $h$-refined sequence of partitions. We therefore opt for the symmetric condition presented in Definition \ref{definition:h_refined} purely for the sake of simplicity.

Given a refined sequence of partitions $\{\calP'_n\}_{n \geq 0}$, whose existence is guaranteed by Lemma \ref{lemma-technical.1}, it is straightforward to construct an $h$-refined sequence.

\begin{lemma}\label{lemma-technical.2}
There always exist $h$-refined sequences of partitions of $X$.
\end{lemma}
\begin{proof}
Let $\{\calP'_n\}_{n \geq 0}$ be a refined sequence of partitions as provided by Lemma \ref{lemma-technical.1}. We define a new sequence of partitions inductively by setting:
\begin{align*}
&\calP_0 := \calP'_0;\\
&\calP_{n+1} := \calP'_{n+1} \wedge \calP_n \wedge h(\calP_n) \wedge h^{-1}(\calP_n) \quad \text{ for all } n \geq 0.
\end{align*}
Because $\{\calP'_n\}_{n \geq 0}$ satisfies conditions (b) and (c) of Lemma \ref{lemma-technical.1}, and $\calP_n \precsim \calP'_n$ for all $n \geq 0$, the new sequence $\{\calP_n\}_{n \geq 0}$ inherits all conditions (a), (b) and (c) of Lemma \ref{lemma-technical.1} and is therefore a refining sequence of partitions of $X$. The condition of being $h$-refined follows immediately from the construction of the sequence.
\end{proof}

We now introduce the main examples that will be used to illustrate the construction detailed in Section \ref{section-the.construction}.

\begin{example}[The two-sided shift]\label{example:two_sided_shift}
In this example, we take $X$ to be the Cantor space, that is, the unique (up to homeomorphism) totally disconnected compact metric space without isolated points. The concrete model we will use is the set of two-sided infinite sequences of zeros and ones, namely
$$X = \{x = (x_i)_{i \in \Z} \mid x_i \in \{0,1\} \text{ for all } i \in \Z\} = \{0,1\}^{\Z}.$$
The topology of $X$ is the usual product topology, where $\{0,1\}$ is given the discrete topology. A basis for this topology is given by the so-called cylinder sets: given a finite tuple $\alpha = (\alpha_1, \dots, \alpha_N) \in \{0,1\}^N$ of zero's and one's and $l \in \Z$, we define the cylinder set $[\alpha;l]$ to be the set of all points $x \in X$ such that the finite segment $(x_l, \dots, x_{l+N-1})$ coincides with $\alpha$. A more convenient notation that we will use for cylinder sets is the following. For integers $r,l \geq 0$ and $\epsilon_{-l}, \dots, \epsilon_r \in \{0,1\}$, we set
$$[\epsilon_{-l} \cdots \underline{\epsilon_0} \cdots \epsilon_r] := [(\epsilon_{-l}, \dots, \epsilon_r);-l] = \{x \in X \mid x_i = \epsilon_i \text{ for all } i \in \{-l, \dots, r\}\}.$$
Note that we symbolize the zeroth position of the index set $\Z$ with an underline. For example, the cylinder set $[01\underline{0}]$ consists of those sequences $x \in X$ for which $x_{-2} = 0$, $x_{-1} = 1$, and $x_0 = 0$.

Strictly speaking, even though this new notation also provides a basis for the topology of $X$, it does not exhaust all possible cylinder sets previously defined as $[\alpha;l]$. The reason is that, in this new notation, we are restricted to finite tuples of zeros and ones that always cover the zeroth position. To overcome this, we extend the allowed values to include a wildcard symbol $\ast$ that represents any possible value of $\{0,1\}$. Thus, for instance, the union of the cylinder sets $[0\underline{1}]$ and $[0\underline{0}]$ is precisely the set of those sequences $x \in X$ for which $x_{-1} = 0$, and we denote this by $[0\underline{\ast}]$. Similarly, $[0100{\ast}{\ast}\underline{\ast}]$ represents the set of sequences $x \in X$ for which $x_{-6} = 0$, $x_{-5} = 1$, $x_{-4} = 0$, and $x_{-3} = 0$.

The metric witnessing the metrizability of the topology defined above is given as follows: for two distinct sequences $x, y \in X$, we set
$$d(x,y) = 2^{-N},$$
where $N$ is the largest non-negative integer such that $x_i = y_i$ for all indices $i$ with $|i| < N$. In other words, $N$ represents the first index (in absolute value) where $x$ and $y$ differ, when starting from the zeroth position and comparing symmetrically outwards.

The homeomorphism is defined here to be the classical two-sided shift, namely
$$\sigma \colon X \to X, \quad \sigma(x)_i = x_{i+1}.$$
It is clear that $\sigma$ is indeed a homeomorphism, as it is bijective and maps cylinder sets to cylinder sets. For instance, $\sigma([00\underline{1}0]) = [001\underline{0}]$ and $\sigma^{-1}([00\underline{1}0]) = [0\underline{0}10]$. Furthermore, $\sigma([10\underline{1}])$ yields the set $[101\underline{\ast}]$.

We can now construct a refined sequence of partitions for $X$. We set $\calP_0 = \{X\}$, $\calP_1 =\{[\underline{0}], [\underline{1}]\}$, $\calP_2 = \{[0\underline{0}0], [0\underline{0}1], [0\underline{1}0], [0\underline{1}1], [1\underline{0}0], [1\underline{0}1], [1\underline{1}0], [1\underline{1}1]\}$ and, more generally, for $n \geq 0$,
$$\calP_{n+1} = \{[\epsilon_{-n} \cdots \underline{\epsilon_0} \cdots \epsilon_n] \mid \epsilon_i \in \{0,1\} \text{ for all } i \in \{-n, \dots, n\}\}.$$
Note that for every block $Z \in \calP_n$, with $n \geq 1$, we have $\operatorname{diam}(Z) = 2^{-n}$; hence, the maximum diameter $\max\{\operatorname{diam}(Z) \mid Z \in \calP_n\}$ tends to zero as $n \to \infty$. Moreover, any cylinder set can be written as a finite union of elements from $\bigcup_{n \geq 0} \calP_n$, so the latter forms a basis for the topology. Furthermore, every block $Z \in \calP_{n+1}$ is uniquely contained in a block $Z' \in \calP_n$ by construction. For example, the block $[01\underline{1}01] \in \calP_3$ is uniquely contained in $[1\underline{1}0] \in \calP_2$. In fact, each block $Z \in \calP_n$ splits into four blocks in $\calP_{n+1}$ for $n \geq 1$. Figure \ref{figure-first.example.1} depicts the first four partitions $\calP_0, \calP_1, \calP_2$, and $\calP_3$, illustrating that each partition is a (strict) refinement of the previous one.

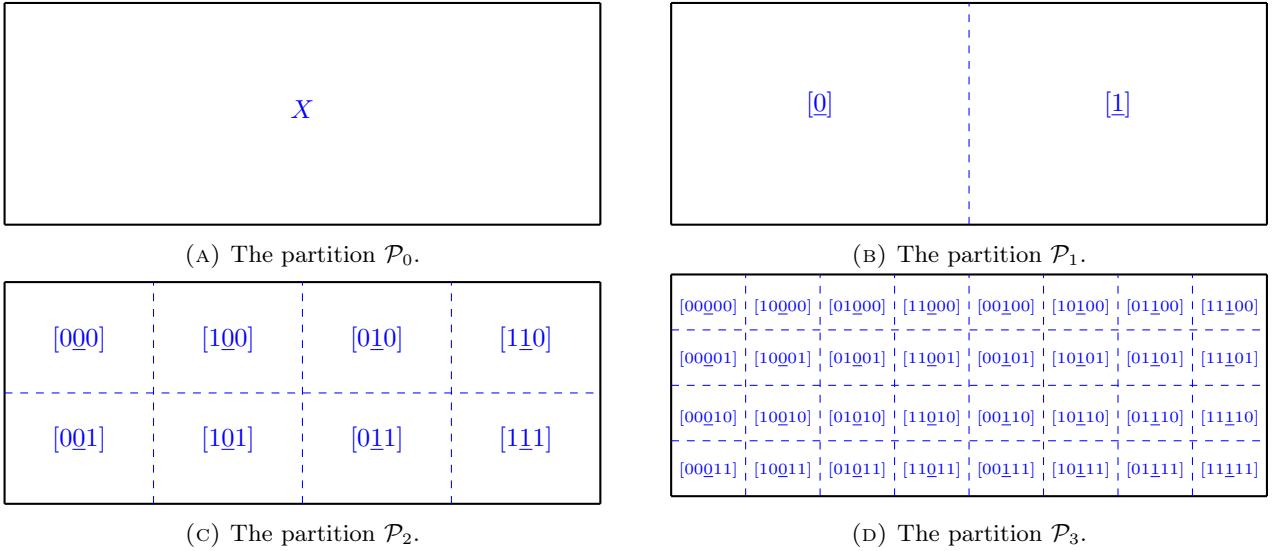
\begin{figure}[htbp]
    \centering
    \begin{subfigure}[b]{0.48\textwidth}
        \centering
        \resizebox{\linewidth}{!}{
            \begin{tikzpicture}
                \draw[black,thick] (-4,0) -- (4,0);
                \draw[black,thick] (-4,0) -- (-4,3);
                \draw[black,thick] (4,0) -- (4,3);
                \draw[black,thick] (-4,3) -- (4,3);
                \node[label=\textcolor{blue}{$X$}] at (0,1.2) {};
            \end{tikzpicture}
        }
        \caption{The partition $\calP_0$.}
    \end{subfigure}
    \hfill
    \begin{subfigure}[b]{0.48\textwidth}
        \centering
        \resizebox{\linewidth}{!}{
            \begin{tikzpicture}
                \draw[black,thick] (-4,0) -- (4,0);
                \draw[black,thick] (-4,0) -- (-4,3);
                \draw[black,thick] (4,0) -- (4,3);
                \draw[black,thick] (-4,3) -- (4,3);
                \draw[blue,dashed] (0,0) -- (0,3);
                \node[label=\textcolor{blue}{$[\underline{0}]$}] at (-2,1.2) {};
                \node[label=\textcolor{blue}{$[\underline{1}]$}] at (2,1.2) {};
            \end{tikzpicture}
        }
        \caption{The partition $\calP_1$.}
    \end{subfigure}
    \begin{subfigure}[b]{0.48\textwidth}
        \centering
        \resizebox{\linewidth}{!}{
            \begin{tikzpicture}
                \draw[black,thick] (-4,0) -- (4,0);
                \draw[black,thick] (-4,0) -- (-4,3);
                \draw[black,thick] (4,0) -- (4,3);
                \draw[black,thick] (-4,3) -- (4,3);
                \draw[blue,dashed] (0,0) -- (0,3);
                \draw[blue,dashed] (-4,1.5) -- (4,1.5);
                \draw[blue,dashed] (-2,0) -- (-2,3);
                \draw[blue,dashed] (2,0) -- (2,3);
                \node[label=\textcolor{blue}{$[0\underline{0}1]$}] at (-3,0.45) {};
                \node[label=\textcolor{blue}{$[0\underline{0}0]$}] at (-3,1.8) {};
                \node[label=\textcolor{blue}{$[1\underline{0}1]$}] at (-1,0.45) {};
                \node[label=\textcolor{blue}{$[1\underline{0}0]$}] at (-1,1.8) {};
                \node[label=\textcolor{blue}{$[0\underline{1}1]$}] at (1,0.45) {};
                \node[label=\textcolor{blue}{$[0\underline{1}0]$}] at (1,1.8) {};
                \node[label=\textcolor{blue}{$[1\underline{1}1]$}] at (3,0.45) {};
                \node[label=\textcolor{blue}{$[1\underline{1}0]$}] at (3,1.8) {};
            \end{tikzpicture}
        }
        \caption{The partition $\calP_2$.}
    \end{subfigure}
    \hfill
    \begin{subfigure}[b]{0.48\textwidth}
        \centering
        \resizebox{\linewidth}{!}{
            \begin{tikzpicture}
                \draw[black,thick] (-4,0) -- (4,0);
                \draw[black,thick] (-4,0) -- (-4,3);
                \draw[black,thick] (4,0) -- (4,3);
                \draw[black,thick] (-4,3) -- (4,3);
                \draw[blue,dashed] (0,0) -- (0,3);
                \draw[blue,dashed] (-4,1.5) -- (4,1.5);
                \draw[blue,dashed] (-2,0) -- (-2,3);
                \draw[blue,dashed] (2,0) -- (2,3);
				\draw[blue,dashed] (-4,0.75) -- (4,0.75);
                \draw[blue,dashed] (-4,2.25) -- (4,2.25);
				\draw[blue,dashed] (-3,0) -- (-3,3);
                \draw[blue,dashed] (-1,0) -- (-1,3);
				\draw[blue,dashed] (1,0) -- (1,3);
                \draw[blue,dashed] (3,0) -- (3,3);
                \node[label=\tiny \textcolor{blue}{$[00\underline{0}00]$}] at (-3.5,2.2) {};
                \node[label=\tiny \textcolor{blue}{$[10\underline{0}00]$}] at (-2.5,2.2) {};
				\node[label=\tiny \textcolor{blue}{$[01\underline{0}00]$}] at (-1.5,2.2) {};
				\node[label=\tiny \textcolor{blue}{$[11\underline{0}00]$}] at (-0.5,2.2) {};
				\node[label=\tiny \textcolor{blue}{$[00\underline{1}00]$}] at (0.5,2.2) {};
				\node[label=\tiny \textcolor{blue}{$[10\underline{1}00]$}] at (1.5,2.2) {};
				\node[label=\tiny \textcolor{blue}{$[01\underline{1}00]$}] at (2.5,2.2) {};
				\node[label=\tiny \textcolor{blue}{$[11\underline{1}00]$}] at (3.5,2.2) {};
				\node[label=\tiny \textcolor{blue}{$[00\underline{0}01]$}] at (-3.5,1.5) {};
                \node[label=\tiny \textcolor{blue}{$[10\underline{0}01]$}] at (-2.5,1.5) {};
				\node[label=\tiny \textcolor{blue}{$[01\underline{0}01]$}] at (-1.5,1.5) {};
				\node[label=\tiny \textcolor{blue}{$[11\underline{0}01]$}] at (-0.5,1.5) {};
				\node[label=\tiny \textcolor{blue}{$[00\underline{1}01]$}] at (0.5,1.5) {};
				\node[label=\tiny \textcolor{blue}{$[10\underline{1}01]$}] at (1.5,1.5) {};
				\node[label=\tiny \textcolor{blue}{$[01\underline{1}01]$}] at (2.5,1.5) {};
				\node[label=\tiny \textcolor{blue}{$[11\underline{1}01]$}] at (3.5,1.5) {};
				\node[label=\tiny \textcolor{blue}{$[00\underline{0}10]$}] at (-3.5,0.7) {};
                \node[label=\tiny \textcolor{blue}{$[10\underline{0}10]$}] at (-2.5,0.7) {};
				\node[label=\tiny \textcolor{blue}{$[01\underline{0}10]$}] at (-1.5,0.7) {};
				\node[label=\tiny \textcolor{blue}{$[11\underline{0}10]$}] at (-0.5,0.7) {};
				\node[label=\tiny \textcolor{blue}{$[00\underline{1}10]$}] at (0.5,0.7) {};
				\node[label=\tiny \textcolor{blue}{$[10\underline{1}10]$}] at (1.5,0.7) {};
				\node[label=\tiny \textcolor{blue}{$[01\underline{1}10]$}] at (2.5,0.7) {};
				\node[label=\tiny \textcolor{blue}{$[11\underline{1}10]$}] at (3.5,0.7) {};
				\node[label=\tiny \textcolor{blue}{$[00\underline{0}11]$}] at (-3.5,0.0) {};
                \node[label=\tiny \textcolor{blue}{$[10\underline{0}11]$}] at (-2.5,0.0) {};
				\node[label=\tiny \textcolor{blue}{$[01\underline{0}11]$}] at (-1.5,0.0) {};
				\node[label=\tiny \textcolor{blue}{$[11\underline{0}11]$}] at (-0.5,0.0) {};
			    \node[label=\tiny \textcolor{blue}{$[00\underline{1}11]$}] at (0.5,0.0) {};
				\node[label=\tiny \textcolor{blue}{$[10\underline{1}11]$}] at (1.5,0.0) {};
				\node[label=\tiny \textcolor{blue}{$[01\underline{1}11]$}] at (2.5,0.0) {};
				\node[label=\tiny \textcolor{blue}{$[11\underline{1}11]$}] at (3.5,0.0) {};
            \end{tikzpicture}
        }
        \caption{The partition $\calP_3$.}
    \end{subfigure}
    \caption{Display of the partitions $\calP_0$ through $\calP_3$.}
	\label{figure-first.example.1}
\end{figure}

The sequence of partitions $\{\calP_n\}_{n \geq 0}$ is also $\sigma$-refined. This is straightforward to verify: for every block $Z = [\epsilon_{-n} \cdots \underline{\epsilon_0} \cdots \epsilon_n] \in \calP_{n+1}$ (with $n \geq 1$), we have:
\begin{equation} \label{equation-two.sided.shift}
\begin{split}
\sigma(Z) = [\epsilon_{-n} \cdots \epsilon_0 \underline{\epsilon_1} \cdots \epsilon_n] \subseteq [\epsilon_{-n+2} \cdots \underline{\epsilon_1} \cdots \epsilon_n] \in \calP_n,\\
\sigma^{-1}(Z) = [\epsilon_{-n} \cdots \underline{\epsilon_{-1}} \epsilon_0 \cdots \epsilon_n] \subseteq [\epsilon_{-n} \cdots \underline{\epsilon_{-1}} \cdots \epsilon_{n-2}] \in \calP_n.
\end{split}
\end{equation}
Thus, we have the refinement relation $\calP_{n+1} \precsim \sigma(\calP_n) \wedge \sigma^{-1}(\calP_n)$. Clearly, it also holds that $\calP_1 \precsim \sigma(\calP_0) \wedge \sigma^{-1}(\calP_0) = \calP_0$.
\end{example}

Staying within the symbolic model of the Cantor space, $X = \{0,1\}^\Z$, we can consider an entirely different topological action.

\begin{example}[The bit-wise NOT map]\label{example:bitwise_NOT}
We consider here the homeomorphism
$$\tau \colon X \to X, \quad \tau(x)_i = 1-x_i,$$
which acts by flipping every $0$ to $1$ and vice versa, coordinate-wise. Note that $\tau^2 = \mathrm{id}_X$; consequently, $\tau$ defines an action of the finite group of two elements, $\Z_2$, on $X$ by homeomorphisms.

The same sequence of partitions $\{\calP_n\}_{n \geq 0}$ from Example \ref{example:two_sided_shift} also serves as a $\tau$-refined sequence of partitions. This is because the map $\tau$ preserves both the length and position of the cylinder sets, merely permuting the elements within each partition level. Thus, we have $\tau(\calP_n) = \tau^{-1}(\calP_n) = \calP_n$.
\end{example}

While the previous examples operated on two-sided infinite sequences, restricting our index set to the non-negative integers $\N_0 := \N \cup \{0\}$ yields another fundamental dynamical system.

\begin{example}[The odometer map on the Cantor space]\label{example:odometer_Cantor}
In this example, we shall consider the classical odometer map on the Cantor space. For ease of notation, we take as a model for $X$ the set of one-sided infinite sequences of zeros and ones, namely
$$X = \{x = (x_i)_{i \in \N_0} \mid x_i \in \{0,1\} \text{ for all } i \in \N_0\} = \{0,1\}^{\N_0}.$$
The topology of $X$ is again the usual product topology, where $\{0,1\}$ is given the discrete topology. A basis for the topology is again given by the cylinder sets, which are now represented by tuples $(\epsilon_0, \dots, \epsilon_r) \in \{0,1\}^{r+1}$ for $r \geq 0$, namely $[\underline{\epsilon_0} \cdots \epsilon_r]$ (thus $l = 0$ if we follow the notation from Example \ref{example:two_sided_shift}). The metric witnessing the metrizability of the topology is defined similarly as before: for two distinct sequences $x, y \in X$, we set $d(x,y) = 2^{-N}$, where $N$ is the largest non-negative integer such that $x_i = y_i$ for all indices $i$ with $0 \leq i < N$.

The homeomorphism is defined as the map $\operatorname{ad} \colon X \to X$ given by
$$\operatorname{ad}(x) = \begin{cases} (1,x_1,x_2,\dots) & \text{if } x_0 = 0, \\ (0,\stackrel{n+1}{\dots},0,1,x_{n+2},\dots) & \text{if } x_i = 1 \text{ for all } i \in \{0,\dots,n\} \text{ and } x_{n+1} = 0, \\ (0,0,\dots) & \text{if } x_i = 1 \text{ for all } i \in \N_0. \end{cases}$$
This map is often called the \textit{adding machine} or \textit{odometer} map, as it can be interpreted as adding one to a binary number represented by the sequence $x$, with carry-over. It is straightforward to verify that $\operatorname{ad}$ is indeed a homeomorphism.

The sequence of partitions $\{\calP_n\}_{n \geq 0}$ defined by $\calP_0 = \{X\}$ and by
$$\calP_{n+1} = \{[\underline{\epsilon_0} \cdots \epsilon_n] \mid \epsilon_i \in \{0,1\} \text{ for all } i \in \{0,\dots,n\}\}$$
for $n \geq 0$, is $\operatorname{ad}$-refined. The reasoning is the same as in Example \ref{example:bitwise_NOT}: the map $\operatorname{ad}$ preserves the length and position of the cylinder sets, merely permuting the elements within each partition level $\calP_n$. For instance,
$$\operatorname{ad}([\underline{0}0]) = [\underline{1}0], \quad \operatorname{ad}([\underline{1}0]) = [\underline{0}1], \quad \operatorname{ad}([\underline{0}1]) = [\underline{1}1] \quad \text{ and } \quad \operatorname{ad}([\underline{1}1]) = [\underline{0}0].$$
\end{example}

While our previous examples focused on Cantor dynamics, the combinatorial encoding also applies to spaces containing isolated points. We conclude this section with such a case.

\begin{example}[The shift map on the one-point compactification of the integers]\label{example:compactified_odometer}
Our last example is defined on a totally disconnected compact metric space containing \emph{infinitely many isolated points}. This example was already considered in \cite{HPS1992} as an example of an essentially minimal dynamical system (we defer the formal definition of essential minimality to Section \ref{subsection:ess_min}, see Definition \ref{definition:ess_min}).

Let $X = \Z^{\ast} := \Z \cup \{\infty\}$ be the one-point compactification of the integers $\Z$. A basis for a topology on $X$ consists of all singletons $\{n\}$ for $n \in \Z$, together with the sets of the form $V_F := (\Z \setminus F) \cup  \{\infty\}$, where $F$ is any finite subset of $\Z$. For notational convenience, we will write $V_n := V_{\{-n,\dots,n\}}$.

Since this basis consists entirely of clopen sets, $X$ is indeed a totally disconnected space. A metric inducing this topology is given by
$$d(n,m) = \frac{|n-m|}{(1+|n|)(1+|m|)} \quad \text{ for } n,m \in \Z, \qquad d(n,\infty) = \frac{1}{1+|n|} \quad \text{ for } n \in \Z.$$
Figure \ref{figure-example.4} provides a visualization of the space $X$.

\begin{figure}[htbp]
    \centering
    \begin{tikzpicture}[scale=3]
        \def\R{0.5} 
        \draw[gray!90, thin] (0,0) circle (\R);
        \node[above=0.8pt, red, font=\bfseries] at (90:\R) {$\infty$};
        \fill[red] (90:\R) circle (0.8pt);
        \fill[blue] (-90:\R) circle (0.5pt);
        \node[font=\scriptsize] at (-90:{\R*(1+0.1)}) {$0$};
        \foreach \n in {9, 10, 11, 12, 13, 14, 15, 16, 17, 18, 19, 20, 21, 22, 23, 24, 25, 26, 27, 30, 32, 34, 36, 38, 40, 42, 45, 50, 60} {
            \fill[blue!60] ({-90+180*(\n)/(\n+2)}:\R) circle (0.5pt);
            \fill[blue!60] ({-90-180*(\n)/(\n+2)}:\R) circle (0.5pt);
        }
        \foreach \n in {1, 2, 3, 4, 5, 6, 7} {
            \fill[blue] ({-90+180*(4*\n)/(4*\n+7)}:\R) circle (0.5pt);
            \node[font=\scriptsize] at ({-90+180*(4*\n)/(4*\n+7)}:{\R*(1+0.15)}) {$\n$};
            \fill[blue] ({-90-180*(4*\n)/(4*\n+7)}:\R) circle (0.5pt);
            \node[font=\scriptsize] at ({-90-180*(4*\n)/(4*\n+7)}:{\R*(1+0.2)}) {-$\n$};
        }
        \node[font=\small, gray, align=center] at (0,0) {$X$};
    \end{tikzpicture}
    \caption{The space $X = \Z^{\ast}$. The integers are mapped onto the circle such that they accumulate at the north pole, representing the point $\infty$.}
    \label{figure-example.4}
\end{figure}
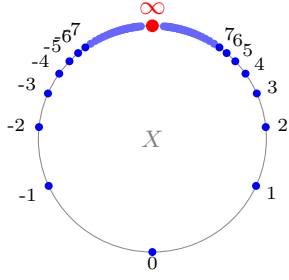

We take $\operatorname{add} \colon X \to X$ to be the homeomorphism defined by
$$\operatorname{add}(n) = n+1 \quad \text{ for } n \in \Z, \qquad \operatorname{add}(\infty) = \infty.$$

We construct a refined sequence of partitions for $X$ as follows. We set $\calP_0 = \{X\}$ and, for $n \geq 0$,
$$\calP_{n+1} = \{\{j\} \mid j \in \{-n,\dots,n\}\} \cup \{V_n\}.$$
Figure \ref{fig:partitions2x3} displays the first six partitions $\calP_0$ through $\calP_5$.
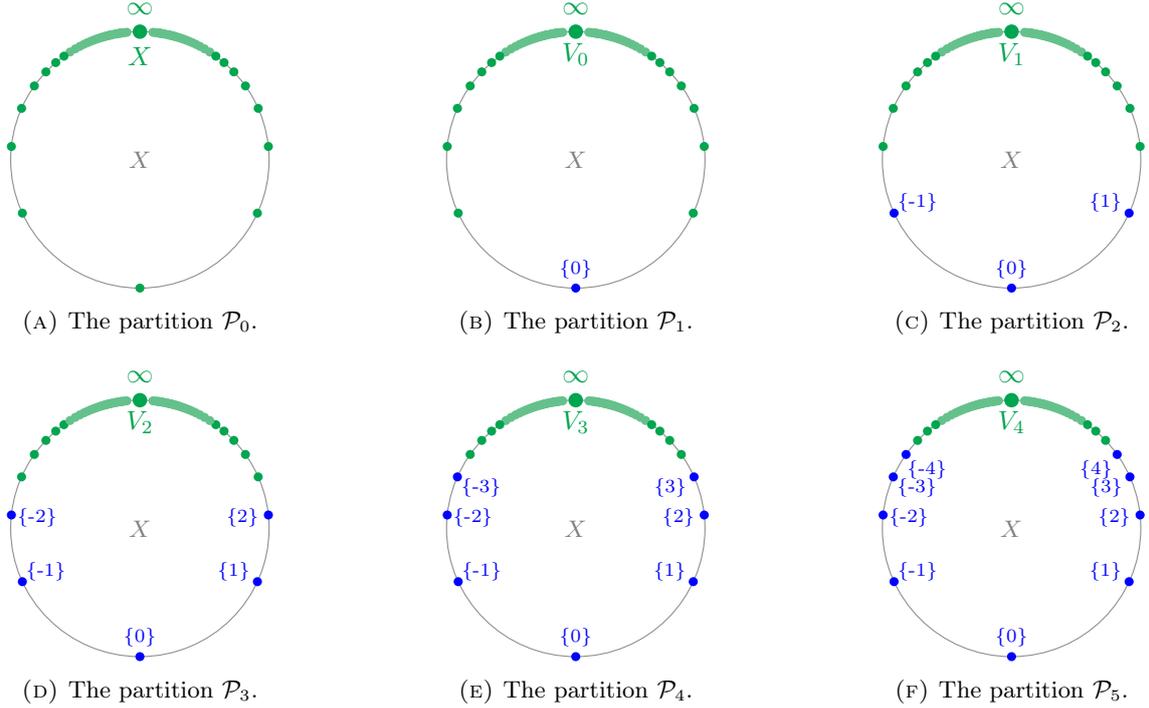
\begin{figure}[htbp]
    \centering
    \begin{subfigure}[b]{0.32\textwidth}
        \centering
        \begin{tikzpicture}[scale=3.4]
         \def\R{0.5} 
         \draw[gray!90, thin] (0,0) circle (\R);
         \node[font=\scriptsize, above=0.8pt, Green, font=\bfseries] at (90:{\R*(1-0.35)}) {$X$};
         \fill[Green] (90:\R) circle (0.8pt);
         \fill[Green] (-90:\R) circle (0.5pt);
         \foreach \n in {9, 10, 11, 12, 13, 14, 15, 16, 17, 18, 19, 20, 21, 22, 23, 24, 25, 26, 27, 30, 32, 34, 36, 38, 40, 42, 45, 50, 60} {
             \fill[Green!60] ({-90+180*(\n)/(\n+2)}:\R) circle (0.5pt);
             \fill[Green!60] ({-90-180*(\n)/(\n+2)}:\R) circle (0.5pt);
         }
         \foreach \n in {1, 2, 3, 4, 5, 6, 7} {
             \fill[Green] ({-90+180*(4*\n)/(4*\n+7)}:\R) circle (0.5pt);
             \fill[Green] ({-90-180*(4*\n)/(4*\n+7)}:\R) circle (0.5pt);
         }
         \node[font=\small, gray, align=center] at (0,0) {$X$};
         \node[Green, align=center] at (0,0.59) {$\infty$};
        \end{tikzpicture}
        \caption{The partition $\calP_0$.}
    \end{subfigure}
    \hfill
    \begin{subfigure}[b]{0.32\textwidth}
        \centering
        \begin{tikzpicture}[scale=3.4]
         \def\R{0.5} 
         \draw[gray!90, thin] (0,0) circle (\R);
         \node[font=\scriptsize, above=0.8pt, Green, font=\bfseries] at (90:{\R*(1-0.35)}) {$V_0$};
         \fill[Green] (90:\R) circle (0.8pt);
         \fill[blue] (-90:\R) circle (0.5pt);
         \node[font=\scriptsize, blue] at (-90:{\R*(1-0.15)}) {$\{0\}$};
         \foreach \n in {9, 10, 11, 12, 13, 14, 15, 16, 17, 18, 19, 20, 21, 22, 23, 24, 25, 26, 27, 30, 32, 34, 36, 38, 40, 42, 45, 50, 60} {
             \fill[Green!60] ({-90+180*(\n)/(\n+2)}:\R) circle (0.5pt);
             \fill[Green!60] ({-90-180*(\n)/(\n+2)}:\R) circle (0.5pt);
         }
         \foreach \n in {1, 2, 3, 4, 5, 6, 7} {
             \fill[Green] ({-90+180*(4*\n)/(4*\n+7)}:\R) circle (0.5pt);
             \fill[Green] ({-90-180*(4*\n)/(4*\n+7)}:\R) circle (0.5pt);
         }
         \node[font=\small, gray, align=center] at (0,0) {$X$};
         \node[Green, align=center] at (0,0.59) {$\infty$};
        \end{tikzpicture}
        \caption{The partition $\calP_1$.}
    \end{subfigure}
    \hfill
    \begin{subfigure}[b]{0.32\textwidth}
        \centering
        \begin{tikzpicture}[scale=3.4]
         \def\R{0.5} 
         \draw[gray!90, thin] (0,0) circle (\R);
         \node[font=\scriptsize, above=0.8pt, Green, font=\bfseries] at (90:{\R*(1-0.35)}) {$V_1$};
         \fill[Green] (90:\R) circle (0.8pt);
         \fill[blue] (-90:\R) circle (0.5pt);
         \node[font=\scriptsize, blue] at (-90:{\R*(1-0.15)}) {$\{0\}$};
         \foreach \n in {9, 10, 11, 12, 13, 14, 15, 16, 17, 18, 19, 20, 21, 22, 23, 24, 25, 26, 27, 30, 32, 34, 36, 38, 40, 42, 45, 50, 60} {
             \fill[Green!60] ({-90+180*(\n)/(\n+2)}:\R) circle (0.5pt);
             \fill[Green!60] ({-90-180*(\n)/(\n+2)}:\R) circle (0.5pt);
         }
         \foreach \n in {1} {
             \fill[blue] ({-90+180*(4*\n)/(4*\n+7)}:\R) circle (0.5pt);
             \node[font=\scriptsize, blue] at ({-90+180*(4*\n)/(4*\n+7)}:{\R*(1-0.2)}) {$\{\n\}$};
             \fill[blue] ({-90-180*(4*\n)/(4*\n+7)}:\R) circle (0.5pt);
             \node[font=\scriptsize, blue] at ({-90-180*(4*\n)/(4*\n+7)}:{\R*(1-0.2)}) {$\{$-$\n\}$};
         }
         \foreach \n in {2, 3, 4, 5, 6, 7} {
             \fill[Green] ({-90+180*(4*\n)/(4*\n+7)}:\R) circle (0.5pt);
             \fill[Green] ({-90-180*(4*\n)/(4*\n+7)}:\R) circle (0.5pt);
         }
         \node[font=\small, gray, align=center] at (0,0) {$X$};
         \node[Green, align=center] at (0,0.59) {$\infty$};
        \end{tikzpicture}
        \caption{The partition $\calP_2$.}
    \end{subfigure}
    \par\bigskip
    \begin{subfigure}[b]{0.32\textwidth}
        \centering
        \begin{tikzpicture}[scale=3.4]
         \def\R{0.5} 
         \draw[gray!90, thin] (0,0) circle (\R);
         \node[font=\scriptsize, above=0.8pt, Green, font=\bfseries] at (90:{\R*(1-0.35)}) {$V_2$};
         \fill[Green] (90:\R) circle (0.8pt);
         \fill[blue] (-90:\R) circle (0.5pt);
         \node[font=\scriptsize, blue] at (-90:{\R*(1-0.15)}) {$\{0\}$};
         \foreach \n in {1, 2} {
             \fill[blue] ({-90+180*(4*\n)/(4*\n+7)}:\R) circle (0.5pt);
             \node[font=\scriptsize, blue] at ({-90+180*(4*\n)/(4*\n+7)}:{\R*(1-0.2)}) {\{$\n\}$};
             \fill[blue] ({-90-180*(4*\n)/(4*\n+7)}:\R) circle (0.5pt);
             \node[font=\scriptsize, blue] at ({-90-180*(4*\n)/(4*\n+7)}:{\R*(1-0.2)}) {$\{$-$\n\}$};
         }
         \foreach \n in {3, 4, 5, 6, 7} {
             \fill[Green] ({-90+180*(4*\n)/(4*\n+7)}:\R) circle (0.5pt);
             \fill[Green] ({-90-180*(4*\n)/(4*\n+7)}:\R) circle (0.5pt);
         }
         \foreach \n in {9, 10, 11, 12, 13, 14, 15, 16, 17, 18, 19, 20, 21, 22, 23, 24, 25, 26, 27, 30, 32, 34, 36, 38, 40, 42, 45, 50, 60} {
             \fill[Green!60] ({-90+180*(\n)/(\n+2)}:\R) circle (0.5pt);
             \fill[Green!60] ({-90-180*(\n)/(\n+2)}:\R) circle (0.5pt);
         }
         \node[font=\small, gray, align=center] at (0,0) {$X$};
         \node[Green, align=center] at (0,0.59) {$\infty$};
        \end{tikzpicture}
        \caption{The partition $\calP_3$.}
    \end{subfigure}
    \hfill
    \begin{subfigure}[b]{0.32\textwidth}
        \centering
        \begin{tikzpicture}[scale=3.4]
         \def\R{0.5} 
         \draw[gray!90, thin] (0,0) circle (\R);
         \node[font=\scriptsize, above=0.8pt, Green, font=\bfseries] at (90:{\R*(1-0.35)}) {$V_3$};
         \fill[Green] (90:\R) circle (0.8pt);
         \fill[blue] (-90:\R) circle (0.5pt);
         \node[font=\scriptsize, blue] at (-90:{\R*(1-0.15)}) {$\{0\}$};
         \foreach \n in {9, 10, 11, 12, 13, 14, 15, 16, 17, 18, 19, 20, 21, 22, 23, 24, 25, 26, 27, 30, 32, 34, 36, 38, 40, 42, 45, 50, 60} {
             \fill[Green!60] ({-90+180*(\n)/(\n+2)}:\R) circle (0.5pt);
             \fill[Green!60] ({-90-180*(\n)/(\n+2)}:\R) circle (0.5pt);
         }
         \foreach \n in {1, 2, 3} {
             \fill[blue] ({-90+180*(4*\n)/(4*\n+7)}:\R) circle (0.5pt);
             \node[font=\scriptsize, blue] at ({-90+180*(4*\n)/(4*\n+7)}:{\R*(1-0.2)}) {$\{\n\}$};
             \fill[blue] ({-90-180*(4*\n)/(4*\n+7)}:\R) circle (0.5pt);
             \node[font=\scriptsize, blue] at ({-90-180*(4*\n)/(4*\n+7)}:{\R*(1-0.2)}) {$\{$-$\n\}$};
         }
         \foreach \n in {4, 5, 6, 7} {
             \fill[Green] ({-90+180*(4*\n)/(4*\n+7)}:\R) circle (0.5pt);
             \fill[Green] ({-90-180*(4*\n)/(4*\n+7)}:\R) circle (0.5pt);
         }
         \node[font=\small, gray, align=center] at (0,0) {$X$};
         \node[Green, align=center] at (0,0.59) {$\infty$};
        \end{tikzpicture}
        \caption{The partition $\calP_4$.}
    \end{subfigure}
    \hfill
    \begin{subfigure}[b]{0.32\textwidth}
        \centering
        \begin{tikzpicture}[scale=3.4]
         \def\R{0.5} 
         \draw[gray!90, thin] (0,0) circle (\R);
         \node[font=\scriptsize, above=0.8pt, Green, font=\bfseries] at (90:{\R*(1-0.35)}) {$V_4$};
         \fill[Green] (90:\R) circle (0.8pt);
         \fill[blue] (-90:\R) circle (0.5pt);
         \node[font=\scriptsize, blue] at (-90:{\R*(1-0.15)}) {$\{0\}$};
         \foreach \n in {9, 10, 11, 12, 13, 14, 15, 16, 17, 18, 19, 20, 21, 22, 23, 24, 25, 26, 27, 30, 32, 34, 36, 38, 40, 42, 45, 50, 60} {
             \fill[Green!60] ({-90+180*(\n)/(\n+2)}:\R) circle (0.5pt);
             \fill[Green!60] ({-90-180*(\n)/(\n+2)}:\R) circle (0.5pt);
         }
         \foreach \n in {1, 2, 3, 4} {
             \fill[blue] ({-90+180*(4*\n)/(4*\n+7)}:\R) circle (0.5pt);
             \node[font=\scriptsize, blue] at ({-90+180*(4*\n)/(4*\n+7)}:{\R*(1-0.2)}) {$\{\n\}$};
             \fill[blue] ({-90-180*(4*\n)/(4*\n+7)}:\R) circle (0.5pt);
             \node[font=\scriptsize, blue] at ({-90-180*(4*\n)/(4*\n+7)}:{\R*(1-0.2)}) {$\{$-$\n\}$};
         }
         \foreach \n in {5, 6, 7} {
             \fill[Green] ({-90+180*(4*\n)/(4*\n+7)}:\R) circle (0.5pt);
             \fill[Green] ({-90-180*(4*\n)/(4*\n+7)}:\R) circle (0.5pt);
         }
         \node[font=\small, gray, align=center] at (0,0) {$X$};
         \node[Green, align=center] at (0,0.59) {$\infty$};
        \end{tikzpicture}
        \caption{The partition $\calP_5$.}
    \end{subfigure}
    \caption{The first six partitions $\calP_0$ through $\calP_5$.}
    \label{fig:partitions2x3}
\end{figure}

That this sequence satisfies conditions (a) and (b) from Lemma \ref{lemma-technical.1} is straightforward. For condition (c), note that $\operatorname{diam}(\{j\}) = 0$ for any $j \in \Z$, and that the maximum distance between points in $V_n$ occurs between $n+1$ and $-n-1$. Thus,
$$\operatorname{diam}(V_n) = d(n+1,-n-1) = \frac{2n+2}{(2+n)^2},$$
which tends to zero as $n \to \infty$.

By construction, the sequence $\{\calP_n\}_{n \geq 0}$ is also $\operatorname{add}$-refined. Indeed, we have
$$\operatorname{add}(\calP_n) = \{\{j\} \mid j \in \{-n+2, \dots, n\}\} \cup \{ V_{\{-n+2, \dots, n\}}\}.$$
Observe that the sets $\{-n\},\{-n+1\}$ and $V_n$ are all contained in $V_{\{-n+2, \dots, n\}}$. Thus, $\calP_{n+1} \precsim \operatorname{add}(\calP_n)$ for $n \geq 1$. Similarly, $\calP_{n+1} \precsim \operatorname{add}^{-1}(\calP_n)$. Finally, since $\operatorname{add}(\calP_0) = \operatorname{add}^{-1}(\calP_0) = \calP_0$, the refinement property $\calP_{n+1} \precsim \operatorname{add}(\calP_n) \wedge \operatorname{add}^{-1}(\calP_n)$ readily follows for all $n \geq 0$.
\end{example}

\subsection{Separated Bratteli diagrams}\label{subsection-separated.Bratteli.diagrams}

As outlined in Section \ref{subsection-dynamical.systems.partitions}, the central idea developed in \cite{AC2024} is to approximate the space $X$, as well as the homeomorphism $h$, by means of Bratteli-like \textit{separated} graphs. The core distinction between these and classical Bratteli diagrams is that the edges of separated graphs are effectively partitioned into different ``colors''. This allows for much richer combinatorial structures. For instance, we will use these graphs and their colors to decouple the topological space $X$ from the action of the homeomorphism $h$, while preserving their internal interplay. We primarily employ two colors: blue and red. The blue edges encode the space $X$, while the red edges represent the dynamics induced by $h$.

Before presenting the main construction in Section \ref{section-the.construction}, we devote this section to formalizing the concept of separated (Bratteli) graphs. Our exposition follows \cite{AC2024}, \cite{AE2014}, \cite{AG2012}.

Let $E = (E^0,E^1,r,s)$ be a directed graph, where $E^0$ denotes the set of vertices, $E^1$ the set of edges, and $r, s \colon E^1 \to E^0$ are the range and source maps, respectively. Since all graphs appearing in this paper will be directed, we shall henceforth drop the adjective and refer to them simply as graphs.

\begin{notation}
Let us clarify some usual graph-theoretic terminology. By a \textit{path} in $E$, be mean a directed sequence of edges $\mu := g_0 \, g_1 \, \cdots \, g_n$ (for $n \geq 0$) such that $g_i \in E^1$ and $s(g_i) = r(g_{i+1})$ for all $i$. The \textit{length} of the path is $|\mu| := n+1$. Vertices $v \in E^0$ are regarded as paths of length zero. Paths of length $n \in \N_0$ will often be referred to as $n$-paths.

We are also interested in \textit{infinite} paths. These are represented by an infinite directed sequence of edges $\gamma := g_0 \, g_1 \, \cdots$ such that, for each $n \geq 0$, the truncation $g_0 \, g_1 \, \cdots \, g_n$ is a path of length $n+1$.

The set of finite paths in $E$ is denoted by $\operatorname{Path}_{\text{fin}}(E^1)$, and the set of infinite paths by $\operatorname{Path}_{\infty}(E^1)$.

The range and source maps extend naturally to $\operatorname{Path}_{\text{fin}}(E^1)$ via the following rules: $r(v) := v$ and $s(v) := v$ whenever $v$ is a vertex; for a path $\mu = g_0 \, g_1 \, \cdots \, g_n$ of length $n+1 \geq 1$, we set $r(\mu) := r(g_0)$ and $s(\mu) := s(g_n)$. The range map further extends to $\operatorname{Path}_{\infty}(E^1)$ by setting $r(\gamma) := r(g_0)$ for $\gamma = g_0 \, g_1 \, \cdots$.
\end{notation}

We now introduce the central object of our construction. It is a generalization of a directed graph that can be viewed as an edge-labeled graph, although we will not strictly adopt that perspective here.

\begin{definition}
A \textit{separated graph} is a pair $(E,C)$ consisting of a graph $E$ and a partition $C$ of $E^1$ that decomposes as $C = \bigsqcup_{v \in E^0} C_v$, where each $C_v$ is a partition of the set of incoming edges $r^{-1}(v)$. The set $C$ is called the \textit{separation}, and $C_v$ is the \textit{separation at} $v \in E^0$.
\end{definition}

Note that this convention differs from the one following in, e.g., \cite{AG2011,AG2012}, where partitions of \textit{outgoing edges} ($s^{-1}(v)$) are considered instead.

\begin{examples}\label{examples-separated.graphs}
\text{ }

\begin{enumerate}
\item[1.] Our first example is the motivating graph that led to the theory of separated graphs (see, for instance, \cite{AG2011,AG2012}, \cite{AEK2013}). For two fixed integers $1 \leq m \leq n$, we define a separated graph $(E(m,n),C(m,n))$ as follows. We set $E(m,n)^0 = \{v,w\}$, $E(m,n)^1 = \{\alpha_1,\dots,\alpha_n,\beta_1,\dots,\beta_m\}$ and $C(m,n) = C(m,n)_v = \{R,B\}$, where $R = \{\alpha_1,\dots,\alpha_n\}$ and $B = \{\beta_1,\dots,\beta_m\}$. All edges have range $v$ and source $w$. Here, we partition the set $r^{-1}(v)$ into two blocks, denoted by $R$ (for \textit{red}) and $B$ (for \textit{blue}). We use these colors to visually represent the separation, see for instance Figure \ref{figure-separated.graph.m.n}.

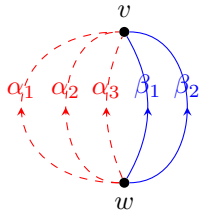
\begin{figure}[H]
    \begin{tikzpicture}
        \node[label=below:$w$,circle,fill=black,scale=0.4] (W) at (0,0) {};
        \node[label=above:$v$,circle,fill=black,scale=0.4] (V) at (0,2) {};
        \draw[dashed,red,postaction={on each segment={mid arrow=red}}] (W) to [out=180,in=180, bend left = 90, looseness=2.2] node[above] {$\alpha_1$} (V);
        \draw[dashed,red,postaction={on each segment={mid arrow=red}}] (W) to [out=90+22.5+22.5,in=180+22.5+22.5, bend left = 90, looseness=1.2] node[above] {$\alpha_2$} (V);
        \draw[dashed,red,postaction={on each segment={mid arrow=red}}] (W) to [out=90+22.5,in=180+22.5+22.5+22.5] node[above] {$\alpha_3$} (V);
        \draw[-,blue,postaction={on each segment={mid arrow=blue}}] (W) to [out=30,in=360-30, bend right = 90, looseness=1.3] node[above] {$\beta_2$} (V);
        \draw[-,blue,postaction={on each segment={mid arrow=blue}}] (W) to [out=60,in=360-60] node[above] {$\beta_1$} (V);
    \end{tikzpicture}
    \caption{The separated graph $(E(2,3),C(2,3))$. The blue edges are solid, while the red edges are dashed.}
    \label{figure-separated.graph.m.n}
\end{figure}

\item[2.] Consider the separated graph $(E,C)$ depicted in Figure \ref{figure-separated.graph.2}. The separation $C = C_{w_1} \sqcup C_{w_2} \sqcup C_{w_3} \sqcup C_{v_1} \sqcup C_{v_2}$ is defined as follows. Since both $w_1$ and $w_2$ are sources (they have no incoming edges), we have $C_{w_1} = C_{w_2} = \emptyset$. For $w_3$, we partition $r^{-1}(w_3)$ into three blocks corresponding to the depicted colors: one block contains the single red edge, another contains the two blue edges, and the third contains the two green edges. These three blocks constitute the partition $C_{w_3}$. Similarly, $C_{v_1}$ consists of two blocks: one block containing the single blue edge, and the other containing the three red edges. Finally, $C_{v_2}$ consists of a single block containing its unique red incoming edge.

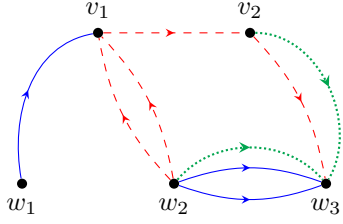
\begin{figure}[H]
    \begin{tikzpicture}
	    \node[circle,fill=black,scale=0.4,label=below:$w_1$] (A1) at (-2,0) {};
	    \node[circle,fill=black,scale=0.4,label=below:$w_2$] (A2) at (0,0) {};
	    \node[circle,fill=black,scale=0.4,label=below:$w_3$] (A3) at (2,0) {};
	    \node[circle,fill=black,scale=0.4,label=above:$v_1$] (B1) at (-1,2) {};
	    \node[circle,fill=black,scale=0.4,label=above:$v_2$] (B2) at (1,2) {};
	    \draw[-,blue,postaction={on each segment={mid arrow=blue}}] (A1) to [out=90+10,in=180+10] (B1);
	    \draw[-,blue,postaction={on each segment={mid arrow=blue}}] (A2) to [out=20,in=180-20] (A3);
	    \draw[-,blue,postaction={on each segment={mid arrow=blue}}] (A2) to [out=360-20,in=180+20] (A3);
	    \draw[dashed,red,postaction={on each segment={mid arrow=red}}] (A2) to [out=180-63.435-16.565,in=360-63.435+16.565] (B1);
	    \draw[dashed,red,postaction={on each segment={mid arrow=red}}] (A2) to [out=180-63.435+16.565,in=270+10] (B1);
	    \draw[dashed,red,postaction={on each segment={mid arrow=red}}] (B2) to [out=360-63.435+16.565,in=180-63.435-16.565] (A3);
        \draw[dashed,red,postaction={on each segment={mid arrow=red}}] (B1) to (B2);
	    \draw[densely dotted,thick,Green,postaction={on each segment={mid arrow=Green}}] (B2) to [out=0,in=60] (A3);
        \draw[densely dotted,thick,Green,postaction={on each segment={mid arrow=Green}}] (A2) to [out=50,in=180-50] (A3);
    \end{tikzpicture}
    \caption{The separated graph $(E,C)$. The blue edges are solid, the red edges are dashed and the green edges are dotted.}
    \label{figure-separated.graph.2}
\end{figure}
\end{enumerate}
\end{examples}

Note in Example \ref{examples-separated.graphs}.2 that although we color certain incoming edges at $v_1$, $v_2$, and $v_3$ red, these edges belong to distinct partitions because they have different ranges. Thus, the use of colors in the diagrams is well-defined and should not lead to confusion. One can think of a separated graph as a graph equipped with a \textit{local} coloring on each set $r^{-1}(v)$.

In this paper, a \textit{bipartite separated graph} is a separated graph $(E,C)$ whose vertex set can be partitioned as $E^0 = E^{0,0} \sqcup E^{0,1}$, with $E^{0,0}, E^{0,1} \neq \emptyset$, such that all edges go from $E^{0,1}$ to $E^{0,0}$, namely $s(E^1) = E^{0,1}$ and  $r(E^1) = E^{0,0}$. Under this definition, the graph from Example \ref{examples-separated.graphs}.1 is bipartite, whereas the graph from Example \ref{examples-separated.graphs}.2 is not.

We now describe the specific type of separated graphs that will appear in our construction.

\begin{definition}\label{definition:separated_Bratteli}\cite[Definition 2.8]{AL2018}
A \textit{separated Bratteli graph} is an infinite separated graph $(F,D)$ satisfying the following properties:
\begin{enumerate}
\item[(a)] The vertex set $F^0$ is the union of finite, non-empty, pairwise disjoint sets $F^{0,j}$, $j \ge 0$.
\item[(b)] The edge set $F^1$ is the union of finite, non-empty, pairwise disjoint sets $F^{1,j}$, $j\ge 0$.
\item[(c)] The range and source maps satisfy $r(F^{1,j})=F^{0,j}$ and $s(F^{1,j})=F^{0,j+1}$ for all $j\ge 0$, respectively.
\end{enumerate}
\end{definition}

Observe that each subgraph $(F^{0,j} \sqcup F^{0,j+1},F^{1,j},r,s)$ is bipartite. Thus, a separated Bratteli diagram is simply a classical Bratteli diagram equipped with a separation $D$. An example of the first three levels of a separated Bratteli diagram is shown in Figure \ref{figure-separated.Bratteli.diagram}.

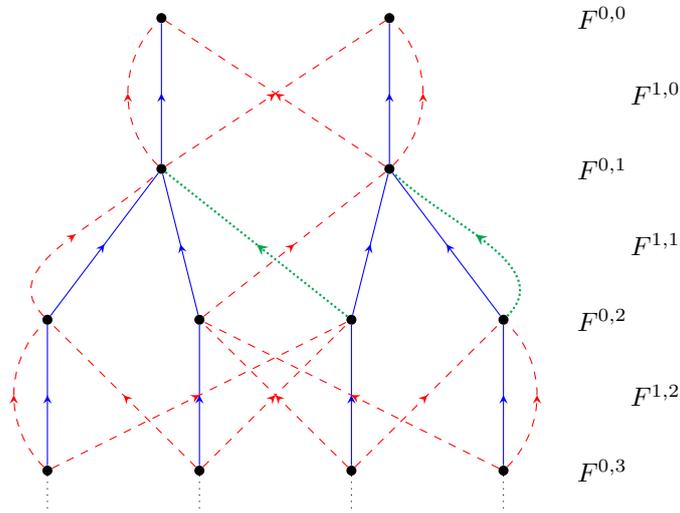
\begin{figure}[H]
\begin{tikzpicture}
	\node[circle,fill=black,scale=0.4] (A1) at (-1.5,3) {};
	\node[circle,fill=black,scale=0.4] (A2) at (1.5,3) {};
	\node[circle,fill=black,scale=0.4] (B1) at (-1.5,1) {};
	\node[circle,fill=black,scale=0.4] (B2) at (1.5,1) {};
	\node[circle,fill=black,scale=0.4] (C1) at (-3,-1) {};
	\node[circle,fill=black,scale=0.4] (C2) at (-1,-1) {};
	\node[circle,fill=black,scale=0.4] (C3) at (1,-1) {};
	\node[circle,fill=black,scale=0.4] (C4) at (3,-1) {};
	\node[circle,fill=black,scale=0.4] (D1) at (-3,-3) {};
	\node[circle,fill=black,scale=0.4] (D2) at (-1,-3) {};
	\node[circle,fill=black,scale=0.4] (D3) at (1,-3) {};
	\node[circle,fill=black,scale=0.4] (D4) at (3,-3) {};
	\draw[-,blue,postaction={on each segment={mid arrow=blue}}] (B1) to (A1);
	\draw[-,blue,postaction={on each segment={mid arrow=blue}}] (B2) to (A2);
	\draw[-,blue,postaction={on each segment={mid arrow=blue}}] (C1) to (B1);
	\draw[-,blue,postaction={on each segment={mid arrow=blue}}] (C2) to (B1);
	\draw[-,blue,postaction={on each segment={mid arrow=blue}}] (C3) to (B2);
	\draw[-,blue,postaction={on each segment={mid arrow=blue}}] (C4) to (B2);
	\draw[-,blue,postaction={on each segment={mid arrow=blue}}] (D1) to (C1);
	\draw[-,blue,postaction={on each segment={mid arrow=blue}}] (D2) to (C2);
	\draw[-,blue,postaction={on each segment={mid arrow=blue}}] (D3) to (C3);
	\draw[-,blue,postaction={on each segment={mid arrow=blue}}] (D4) to (C4);
	\draw[dashed,red,postaction={on each segment={mid arrow=red}}] (B1) to [out=135,in=225] (A1);
	\draw[dashed,red,postaction={on each segment={mid arrow=red}}] (B1) to (A2);
	\draw[dashed,red,postaction={on each segment={mid arrow=red}}] (B2) to (A1);
	\draw[dashed,red,postaction={on each segment={mid arrow=red}}] (B2) to [out=45,in=315] (A2);
	\draw[dashed,red,postaction={on each segment={mid arrow=red}}] (C1) to [out=135,in=225] (B1);
	\draw[dashed,red,postaction={on each segment={mid arrow=red}}] (C2) to (B2);
	\draw[dashed,red,postaction={on each segment={mid arrow=red}}] (D1) to [out=135,in=225] (C1);
	\draw[dashed,red,postaction={on each segment={mid arrow=red}}] (D2) to (C1);
	\draw[dashed,red,postaction={on each segment={mid arrow=red}}] (D1) to (C3);
	\draw[dashed,red,postaction={on each segment={mid arrow=red}}] (D2) to (C3);
	\draw[dashed,red,postaction={on each segment={mid arrow=red}}] (D3) to (C2);
	\draw[dashed,red,postaction={on each segment={mid arrow=red}}] (D3) to (C4);
	\draw[dashed,red,postaction={on each segment={mid arrow=red}}] (D4) to (C2);
	\draw[dashed,red,postaction={on each segment={mid arrow=red}}] (D4) to [out=45,in=315] (C4);
	\draw[densely dotted,thick,Green,postaction={on each segment={mid arrow=Green}}] (C3) to (B1);
    \draw[densely dotted,thick,Green,postaction={on each segment={mid arrow=Green}}] (C4) to [out=45,in=315] (B2);
	\node at (4.3,3){$F^{0,0}$};
	\node at (5,2){$F^{1,0}$};
	\node at (4.3,1){$F^{0,1}$};
	\node at (5,0){$F^{1,1}$};
	\node at (4.3,-1){$F^{0,2}$};
	\node at (5,-2){$F^{1,2}$};
	\node at (4.3,-3){$F^{0,3}$};
	\draw[black,dotted] (-3,-3.5) -- (-3,-3);
	\draw[black,dotted] (-1,-3.5) -- (-1,-3);
	\draw[black,dotted] (1,-3.5) -- (1,-3);
	\draw[black,dotted] (3,-3.5) -- (3,-3);
\end{tikzpicture}
\caption{The first three levels of a separated Bratteli diagram. The blue edges are solid, the red edges are dashed and the green edges are dotted.}
\label{figure-separated.Bratteli.diagram}
\end{figure}

\section{The separated graph model of a dynamical system}\label{section-the.construction}

This section aims to describe the construction from \cite[Section 3]{AC2024} for the four examples presented in Section \ref{subsection-dynamical.systems.partitions}. To begin, we specialize the construction to our setting: a dynamical system $(X,h)$ consisting of a \textit{homeomorphism} $h$ defined on a totally disconnected compact metric space $X$ (cf. \cite[Section 4]{AC2024}). The crucial ingredient is a fixed sequence of $h$-refined partitions $\{\calP_n\}_{n \geq 0}$, a concept introduced in Section \ref{subsection-dynamical.systems.partitions} and whose existence was established in Lemma \ref{lemma-technical.2}.

As explained at the beginning of Section \ref{subsection-separated.Bratteli.diagrams}, we will make use of \textit{two-colored} separated Bratteli diagrams. The main idea for associating such a diagram $(F,D)$ to $(X,h)$ is the following: the vertices $F^0$ correspond to the sets $Z \in \calP_n$ and the edges $F^1$ correspond to inclusions. For each vertex $Z \in F^0$, we take the separation $D_v$ to be two-colored, using the colors blue and red. More concretely, the blue edges correspond to inclusion maps $Z' \subseteq Z$, and the red edges correspond to the action of $h$ or $h^{-1}$, depending on the parity of the level.

\begin{notation}
We will use the letter $e$ to denote blue edges and the letter $f$ to denote red edges. Thus, for example, the path $\mu = e_0 \, e_1 \, e_2$ is a $3$-path consisting of blue edges.
\end{notation}

\begin{construction}\label{construction-h.diagram}\cite[Construction 3.22]{AC2024}
Fix an $h$-refined sequence of partitions $\{\calP_n\}_{n \geq 0}$ of $X$. We construct $(F,D)$ inductively, as follows.

\begin{enumerate}

\item[(a)] (Vertices) For each $n \geq 0$, set $F^{0,n} := \calP_n$.

\item[(b)] (Edges) For each $n \geq 0$, set $F^{1,n} := B^{(n)} \sqcup R^{(n)}$, where
\begin{align*}
B^{(n)} & := \bigsqcup_{Z \in \calP_n} B_Z, \qquad B_Z := \{e(Z',Z) \mid Z' \in \calP_{n+1} \text{ with } Z' \subseteq Z\}, \\
R^{(n)} & := \bigsqcup_{Z \in \calP_n} R_Z, \qquad R_Z := \{f(Z',Z) \mid Z' \in \calP_{n+1} \text{ with } Z' \subseteq h^{(-1)^n}(Z)\}.
\end{align*}
Thus for any pair $Z \in \calP_n$, $Z' \in \calP_{n+1}$ such that $Z' \subseteq Z$, we write down a single \textit{blue} edge $e(Z',Z)$ with source $Z'$ and range $Z$. Also, for any pair $Z \in \calP_n, Z' \in \calP_{n+1}$ such that $Z' \subseteq h(Z)$ if $n$ is even and $Z' \subseteq h^{-1}(Z)$ if $n$ is odd, we write down a single \textit{red} edge $f(Z',Z)$ with source $Z'$ and range $Z$.
\item[(c)] (Separations) We set $D = \bigsqcup_{Z \in F^0} D_Z$, where for $n \geq 0$, the separation at $Z \in \calP_n$ is defined to be $D_Z = \{B_Z, R_Z\}$.
\end{enumerate}
Figure \ref{figure-separated.Bratteli.diagram.edges} summarizes the schematics of Construction \ref{construction-h.diagram}.
\end{construction}

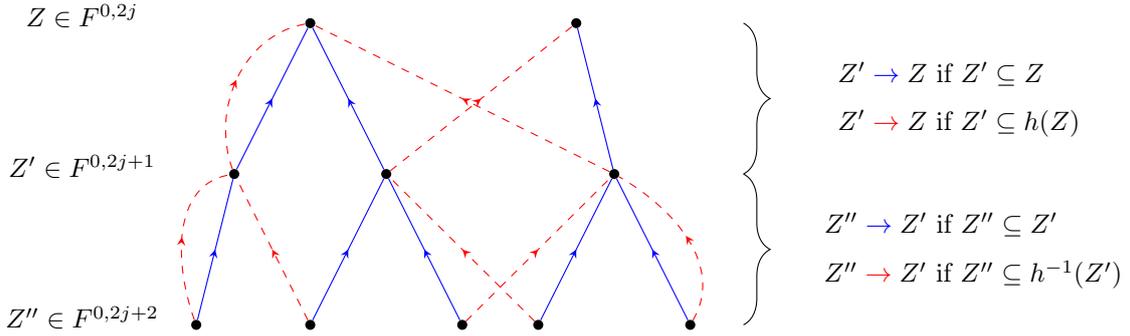
\begin{figure}[H]
\begin{tikzpicture}
	\node[circle,fill=black,scale=0.4] (B1) at (-3.5,0) {};
	\node[circle,fill=black,scale=0.4] (B2) at (-2,0) {};
	\node[circle,fill=black,scale=0.4] (B3) at (0,0) {};
	\node[circle,fill=black,scale=0.4] (B4) at (1,0) {};
	\node[circle,fill=black,scale=0.4] (B5) at (3,0) {};
	\node[circle,fill=black,scale=0.4] (M1) at (-3,2) {};
	\node[circle,fill=black,scale=0.4] (M2) at (-1,2) {};
	\node[circle,fill=black,scale=0.4] (M3) at (2,2) {};
	\node[circle,fill=black,scale=0.4] (T1) at (-2,4) {};
	\node[circle,fill=black,scale=0.4] (T2) at (1.5,4) {};
	\node[label=$Z \in F^{0,2j}$] (2N) at (-5,3.7) {};
	\node[label=$Z' \in F^{0,2j+1}$] (2N+1) at (-5,1.7) {};
	\node[label=$Z'' \in F^{0,2j+2}$] (2N+2) at (-5,-0.3) {};
	%
	\draw[-,blue,postaction={on each segment={mid arrow=blue}}] (M1) to (T1);
	\draw[-,blue,postaction={on each segment={mid arrow=blue}}] (M2) to (T1);
	\draw[-,blue,postaction={on each segment={mid arrow=blue}}] (M3) to (T2);
	\draw[-,blue,postaction={on each segment={mid arrow=blue}}] (B1) to (M1);
	\draw[-,blue,postaction={on each segment={mid arrow=blue}}] (B2) to (M2);
	\draw[-,blue,postaction={on each segment={mid arrow=blue}}] (B3) to (M2);
	\draw[-,blue,postaction={on each segment={mid arrow=blue}}] (B4) to (M3);
	\draw[-,blue,postaction={on each segment={mid arrow=blue}}] (B5) to (M3);
    \draw[dashed,red,postaction={on each segment={mid arrow=red}}] (M1) to [out=110,in=190] (T1);
	\draw[dashed,red,postaction={on each segment={mid arrow=red}}] (M2) to (T2);
	\draw[dashed,red,postaction={on each segment={mid arrow=red}}] (M3) to (T1);
    \draw[dashed,red,postaction={on each segment={mid arrow=red}}] (B1) to [out=110,in=190] (M1);
	\draw[dashed,red,postaction={on each segment={mid arrow=red}}] (B2) to (M1);
    \draw[dashed,red,postaction={on each segment={mid arrow=red}}] (B3) to (M3);
    \draw[dashed,red,postaction={on each segment={mid arrow=red}}] (B4) to (M2);
    \draw[dashed,red,postaction={on each segment={mid arrow=red}}] (B5) to [out=60,in=360-30] (M3);
    \draw[decorate,decoration={brace,amplitude=10pt,mirror}] 
    (3.7,0) -- (3.7,2);
	\node at (6.3,1.33){$Z'' \begingroup\color{blue} \ra \endgroup Z' \text{ if } Z'' \subseteq Z'$};
	\node at (6.722,0.67){$Z'' \begingroup\color{red} \ra \endgroup Z' \text{ if } Z'' \subseteq h^{-1}(Z')$};	
	\draw[decorate,decoration={brace,amplitude=10pt,mirror}] 
    (3.7,2) -- (3.7,4);
	\node at (6.3,3.33){$Z' \begingroup\color{blue} \ra \endgroup Z \text{ if } Z' \subseteq Z$};
	\node at (6.537,2.67){$Z' \begingroup\color{red} \ra \endgroup Z \text{ if } Z' \subseteq h(Z)$};
\end{tikzpicture}
\caption{Construction of the edges of $(F,D)$. The blue edges are solid, while the red edges are dashed.}
\label{figure-separated.Bratteli.diagram.edges}
\end{figure}
Some observations concerning Construction \ref{construction-h.diagram} are in order.

\begin{observations}\label{observations-1}
\text{ }

\begin{enumerate}
\item[(1)] For each $n \geq 1$ and $Z \in \calP_n$, there is exactly one blue edge and one red edge with source $Z$. This is a direct consequence of the fact that $\calP_{n+1}$ is finer than $\calP_n \wedge h(\calP_n) \wedge h^{-1}(\calP_n)$. Hence $|F^{1,n}| = 2|F^{0,n+1}|$ for all $n \geq 0$.
\item[(2)] Suppose that we have a $2$-path $\mu = e_0 \, e_1$ consisting of blue edges, with source $Z'' \in \calP_{n+2}$, range $Z \in \calP_n$, and middle vertex $Z' \in \calP_{n+1}$. Thus $e_0 = e(Z',Z)$ and $e_1 = e(Z'',Z')$. This means precisely that
$$Z'' \subseteq Z' \subseteq Z.$$
For simplicity, we take $n$ to be even. Since $\calP_{n+2}$ is finer than $h^{-1}(\calP_{n+1})$, there exist unique vertices $Z'_1 \in \calP_{n+1}$, $Z_1 \in \calP_n$ such that $Z'' \subseteq h^{-1}(Z'_1)$ and $Z'_1 \subseteq h(Z_1)$. But then
$$Z'' \subseteq h^{-1}(Z'_1) \subseteq h^{-1}(h(Z_1)) = Z_1.$$
Because we already know $Z'' \subseteq Z$ and these are partitions, $Z_1 = Z$ by uniqueness. Therefore, there exists a $2$-path $\mu' = f_0 f_1$ consisting of red edges, with source $Z''$, range $Z$, and middle vertex $Z'_1$. Namely, we take $f_0 := f(Z'_1,Z)$ and $f_1 := f(Z'',Z'_1)$. A similar result holds for odd $n$, and also if we exchange blue $2$-paths for red $2$-paths.
\end{enumerate}
\end{observations}

Observation \ref{observations-1}(2) indicates that the separated Bratteli diagram $(F,D)$ contains a large amount of \textit{rhombi}. These are quadruples of edges $(e_0,e_1;f_0,f_1)$, with $e_i$ blue and $f_i$ red, satisfying the incidence relations
$$r(e_1) = s(e_0), \quad r(f_1) = s(f_0), \quad r(e_0) = r(f_0) \quad \text{and} \quad s(e_1) = s(f_1).$$
See Figure \ref{figure-romb} for an illustration of a rhombus.
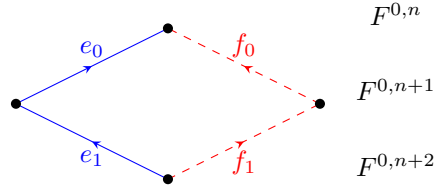
\begin{figure}[H]
\begin{tikzpicture}
	\path [draw=blue,postaction={on each segment={mid arrow=blue}}]
		(0,0) -- node[below,blue]{$e_1$} (-2,1)
		(-2,1) -- node[above,blue]{$e_0$} (0,2)
		;
	\path [dashed,draw=red,postaction={on each segment={mid arrow=red}}]
		(0,0) -- node[below,red]{$f_1$} (2,1)
		(2,1) -- node[above,red]{$f_0$} (0,2)
		;
	\node[circle,fill=black,scale=0.4] (T) at (0,0) {};
    \node[circle,fill=black,scale=0.4] (T) at (-2,1) {};
    \node[circle,fill=black,scale=0.4] (T) at (2,1) {};
    \node[circle,fill=black,scale=0.4] (T) at (0,2) {};
	\node[label=$F^{0,n}$] (N) at (3,1.8) {};
	\node[label=$F^{0,n+1}$] (N+1) at (3,0.8) {};
	\node[label=$F^{0,n+2}$] (N+2) at (3,-0.2) {};
\end{tikzpicture}
\caption{A quadruple of edges $(e_0,e_1;f_0,f_1)$ defining a rhombus in $(F,D)$. The blue edges are solid, and the red edges are dashed.}
\label{figure-romb}
\end{figure}

These observations lead to a characterization of the class of separated Bratteli diagrams that arise from dynamical systems $(X,h)$ via this construction. These are termed \textit{$h$-diagrams} (see \cite[Definition 4.1]{AC2024}). The main result of \cite[Section 4]{AC2024}, Theorem 4.8, establishes a bijective correspondence --- modulo suitable equivalence relations --- between the class of dynamical systems $(X,h)$ considered here and the class of $h$-diagrams.

The rest of this section is devoted to explicitly describing the $h$-diagrams associated with the four examples presented in Section \ref{subsection-dynamical.systems.partitions}.

\subsection{Example: the two-sided shift}\label{example-two.sided.shift}

Recall that in the two-sided shift example, $X = \{0,1\}^{\Z}$ and the homeomorphism considered is the shift $\sigma \colon X \to X$ defined by $\sigma(x)_i = x_{i+1}$. As established in Section \ref{subsection-dynamical.systems.partitions}, the sequence of partitions $\{\calP_n\}_{n \geq 0}$, given by $\calP_0 = \{X\}$ and $\calP_{n+1} = \{[\epsilon_{-n} \cdots \underline{\epsilon_0} \cdots \epsilon_n] \mid \epsilon_i \in \{0,1\} \text{ for all } i\}$ for $n \geq 1$, is $\sigma$-refined. We now utilize this sequence to construct the associated $h$-diagram $(F,D)$.

Set $F^{0,n} = \calP_n$. Note that $|F^{0,0}| = 1$ and $|F^{0,n}| = 2^{2n-1}$ for $n \geq 1$. The zeroth level of the diagram, namely the bipartite graph $(F^{0,0} \sqcup F^{0,1},F^{1,0})$ together with its separations, is easily established: we have a single vertex $X$ in the top row $F^{0,0}$ and two vertices $[\underline{0}]$, $[\underline{1}]$ in the bottom row $F^{0,1}$, and we have two blue edges and two red edges in $F^{1,0}$. The blue edges correspond to the inclusions $[\underline{0}], [\underline{1}] \subseteq X$, and the red edges correspond to the inclusions $[\underline{0}], [\underline{1}] \subseteq \sigma(X) = X$. We can observe that in this case all inclusions are trivial, in the sense that all edges share the same range.

For $n \geq 1$, the $n$-th level $(F^{0,n} \sqcup F^{0,n+1},F^{1,n})$ connects $2^{2n-1}$ vertices in the top row $F^{0,n}$ to $2^{2n+1}$ vertices in the bottom row $F^{0,n+1}$. Accordingly, $F^{1,n}$ contains $2 \cdot 2^{2n+1}$ edges, split evenly between blue and red.

The blue edges represent the inclusions between blocks of the partitions:
$$[\epsilon_{-n} \cdots \underline{\epsilon_0} \cdots \epsilon_n] \subseteq [\epsilon_{-n+1} \cdots \underline{\epsilon_0} \cdots \epsilon_{n-1}].$$
Recalling the computations \eqref{equation-two.sided.shift}, the red edges are determined by the inclusions:
$$[\epsilon_{-n} \cdots \underline{\epsilon_0} \cdots \epsilon_n] \subseteq \sigma([\epsilon_{-n} \cdots \underline{\epsilon_{-1}} \cdots \epsilon_{n-2}])$$
when $n$ is even, and
$$[\epsilon_{-n} \cdots \underline{\epsilon_0} \cdots \epsilon_n] \subseteq \sigma^{-1}([\epsilon_{-n+2} \cdots \underline{\epsilon_1} \cdots \epsilon_n])$$
when $n$ is odd. To illustrate, at the second level ($n=1$), we have the following inclusion relations: $[\epsilon_{-1} \underline{\epsilon_0} \epsilon_1] \subseteq [\underline{\epsilon_0}]$ and $[\epsilon_{-1} \underline{\epsilon_0} \epsilon_1] \subseteq \sigma^{-1}([\underline{\epsilon_1}])$. Diagrammatically, this translates to drawing a blue edge $[\epsilon_{-1}\underline{\epsilon_0}\epsilon_1] \textcolor{blue}{\to} [\underline{\epsilon_0}]$ and a red edge $[\epsilon_{-1}\underline{\epsilon_0}\epsilon_1] \textcolor{red}{\to} [\underline{\epsilon_1}]$.

The first two levels of the resulting $(F,D)$ are depicted in Figure \ref{figure:first_levels_shift}.

\begin{figure}[H]
\begin{tikzpicture}
	\node[circle,fill=black,scale=0.4,label=above:\textcolor{black}{$X$}] (A1) at (0,4) {};
    \node[circle,fill=black,scale=0.4,label=above:\textcolor{black}{$[\underline{0}]$}] (B1) at (-4,2) {};
    \node[circle,fill=black,scale=0.4,label=above:\textcolor{black}{$[\underline{1}]$}] (B2) at (4,2) {};
	\node[circle,fill=black,scale=0.4,label=below:\textcolor{black}{$[0\underline{0}0]$}] (C1) at (-7,0) {};
	\node[circle,fill=black,scale=0.4,label=below:\textcolor{black}{$[1\underline{0}0]$}] (C2) at (-5,0) {};
	\node[circle,fill=black,scale=0.4,label=below:\textcolor{black}{$[0\underline{0}1]$}] (C3) at (-3,0) {};
	\node[circle,fill=black,scale=0.4,label=below:\textcolor{black}{$[1\underline{0}1]$}] (C4) at (-1,0) {};
	\node[circle,fill=black,scale=0.4,label=below:\textcolor{black}{$[0\underline{1}0]$}] (C5) at (1,0) {};
	\node[circle,fill=black,scale=0.4,label=below:\textcolor{black}{$[1\underline{1}0]$}] (C6) at (3,0) {};
	\node[circle,fill=black,scale=0.4,label=below:\textcolor{black}{$[0\underline{1}1]$}] (C7) at (5,0) {};
	\node[circle,fill=black,scale=0.4,label=below:\textcolor{black}{$[1\underline{1}1]$}] (C8) at (7,0) {};
	\draw[-,blue,postaction={on each segment={mid arrow=blue}}] (B1) to (A1);
	\draw[-,blue,postaction={on each segment={mid arrow=blue}}] (B2) to (A1);
	\draw[-,blue,postaction={on each segment={mid arrow=blue}}] (C1) to (B1);
	\draw[-,blue,postaction={on each segment={mid arrow=blue}}] (C2) to (B1);
	\draw[-,blue,postaction={on each segment={mid arrow=blue}}] (C3) to (B1);
	\draw[-,blue,postaction={on each segment={mid arrow=blue}}] (C4) to (B1);
	\draw[-,blue,postaction={on each segment={mid arrow=blue}}] (C5) to (B2);
	\draw[-,blue,postaction={on each segment={mid arrow=blue}}] (C6) to (B2);
	\draw[-,blue,postaction={on each segment={mid arrow=blue}}] (C7) to (B2);
	\draw[-,blue,postaction={on each segment={mid arrow=blue}}] (C8) to (B2);
	\draw[dashed,red,postaction={on each segment={mid arrow=red}}] (B1) to [out=90-10,in=180+10] (A1);
	\draw[dashed,red,postaction={on each segment={mid arrow=red}}] (B2) to [out=90+10,in=360-10] (A1);
	\draw[dashed,red,postaction={on each segment={mid arrow=red}}] (C1) to [out=90+10,in=180+10] (B1);
	\draw[dashed,red,postaction={on each segment={mid arrow=red}}] (C2) to [out=90+10,in=180+60] (B1);
	\draw[dashed,red,postaction={on each segment={mid arrow=red}}] (C3) to (B2);
	\draw[dashed,red,postaction={on each segment={mid arrow=red}}] (C4) to (B2);
	\draw[dashed,red,postaction={on each segment={mid arrow=red}}] (C5) to (B1);
	\draw[dashed,red,postaction={on each segment={mid arrow=red}}] (C6) to (B1);
	\draw[dashed,red,postaction={on each segment={mid arrow=red}}] (C7) to [out=90-10,in=360-60] (B2);
	\draw[dashed,red,postaction={on each segment={mid arrow=red}}] (C8) to [out=90-10,in=360-10] (B2);
	\node at (7.7,4){$\calP_0$};
	\node at (7.7,2){$\calP_1$};
	\node at (7.7,0){$\calP_2$};
	\draw[black,dotted] (-7,-0.8) -- (-7,0);
	\draw[black,dotted] (-5,-0.8) -- (-5,0);
	\draw[black,dotted] (-3,-0.8) -- (-3,0);
	\draw[black,dotted] (-1,-0.8) -- (-1,0);
	\draw[black,dotted] (1,-0.8) -- (1,0);
	\draw[black,dotted] (3,-0.8) -- (3,0);
	\draw[black,dotted] (5,-0.8) -- (5,0);
	\draw[black,dotted] (7,-0.8) -- (7,0);
\end{tikzpicture}
\caption{The first two levels of the $h$-diagram associated to the two-sided shift. The blue edges are solid, and the red edges are dashed.}
\label{figure:first_levels_shift}
\end{figure}
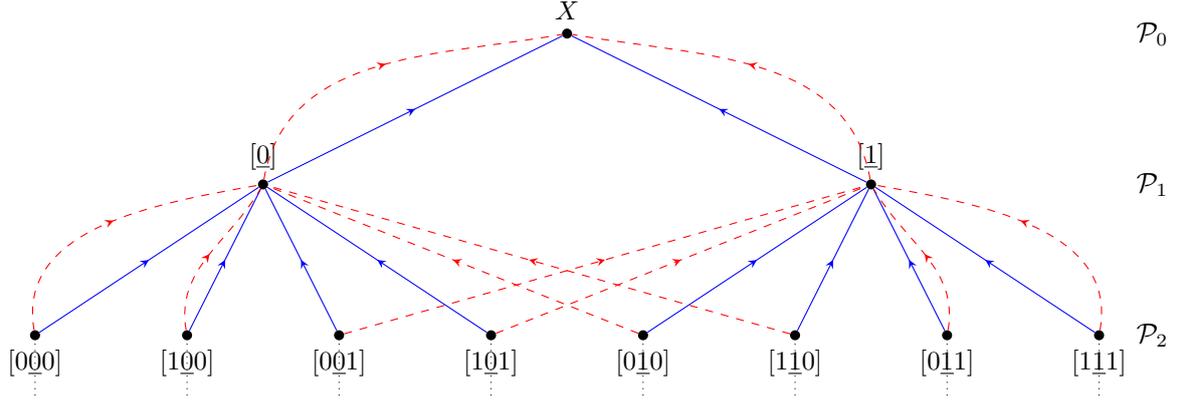

We can readily infer dynamical properties of $(X,\sigma)$ from the diagram. For instance, the shift map $\sigma$ has two fixed points, $(\dots, 0, \underline{0}, 0, \dots)$ and $(\dots, 1, \underline{1}, 1, \dots)$. These invariant sequences manifest in the diagram as straight, isolated vertical paths where the blue and red edges parallel each other (see Figure \ref{figure-fixed.points}).

\begin{figure}[H]
\begin{tikzpicture}
	\node[circle,fill=black,scale=0.4,label=above:\textcolor{black}{$X$}] (A1) at (0,4) {};
    \node[circle,fill=black,scale=0.4,label=below:\textcolor{black}{$[\underline{0}]$}] (B1) at (-3,2.6) {};
    \node[circle,fill=black,scale=0.4,label=below:\textcolor{black}{$[\underline{1}]$}] (B2) at (3,2.6) {};
	\node[circle,fill=black,scale=0.4,label=below:\textcolor{black}{$[0\underline{0}0]$}] (C1) at (-5,1.3) {};
	\node[circle,fill=black,scale=0.4,label=below:\textcolor{black}{$[1\underline{1}1]$}] (C2) at (5,1.3) {};
    \node[circle,fill=black,scale=0.4,label=below:\textcolor{black}{$[00\underline{0}00]$}] (D1) at (-6,0) {};
	\node[circle,fill=black,scale=0.4,label=below:\textcolor{black}{$[11\underline{1}11]$}] (D2) at (6,0) {};
	\draw[-,blue,postaction={on each segment={mid arrow=blue}}] (B1) to (A1);
	\draw[-,blue,postaction={on each segment={mid arrow=blue}}] (B2) to (A1);
	\draw[-,blue,postaction={on each segment={mid arrow=blue}}] (C1) to (B1);
	\draw[-,blue,postaction={on each segment={mid arrow=blue}}] (C2) to (B2);
	\draw[-,blue,postaction={on each segment={mid arrow=blue}}] (D1) to (C1);
	\draw[-,blue,postaction={on each segment={mid arrow=blue}}] (D2) to (C2);
	\draw[dashed,red,postaction={on each segment={mid arrow=red}}] (B1) to [out=90-10,in=180+10] (A1);
	\draw[dashed,red,postaction={on each segment={mid arrow=red}}] (B2) to [out=90+10,in=360-10] (A1);
	\draw[dashed,red,postaction={on each segment={mid arrow=red}}] (C1) to [out=90+10,in=180+10] (B1);
	\draw[dashed,red,postaction={on each segment={mid arrow=red}}] (C2) to [out=90-10,in=360-10] (B2);
	\draw[dashed,red,postaction={on each segment={mid arrow=red}}] (D1) to [out=90+10,in=180+10] (C1);
	\draw[dashed,red,postaction={on each segment={mid arrow=red}}] (D2) to [out=90-10,in=360-10] (C2);
	\node at (7.5,4){$\calP_0$};
	\node at (7.5,2.6){$\calP_1$};
	\node at (7.5,1.3){$\calP_2$};
    \node at (7.5,0){$\calP_3$};
	\draw[black,dotted] (-6,-0.8) -- (-6,0);
	\draw[black,dotted] (6,-0.8) -- (6,0);
\end{tikzpicture}
\caption{Diagrammatic representation of the two fixed points for the two-sided shift. The blue edges are solid, and the red edges are dashed.}
\label{figure-fixed.points}
\end{figure}
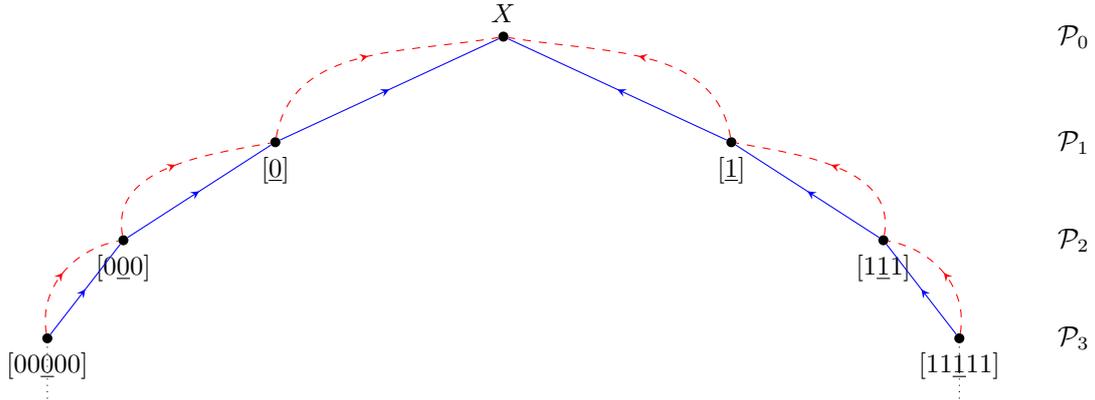

\subsection{Example: the bit-wise NOT map}\label{example-bitwise.NOT}

Returning to the bit-wise NOT map $(X,\tau)$, we retain the symbolic representation for $X$ and the sequence of partitions used for the two-sided shift (Example \ref{example-two.sided.shift}). Consequently, the vertices and the blue edges (representing inclusions between blocks) of the associated $h$-diagram $(F,D)$ are identical to those in the previous example. The edge counts also match: for every $n \geq 0$, there are $2^{2n+1}$ blue edges and $2^{2n+1}$ red edges.

The distinction arises solely in the routing of the red edges for $n \geq 1$. Using the notation $\ol{\epsilon} := 1 - \epsilon$ for $\epsilon \in \{0,1\}$, we have $\tau(x)_i = \ol{x}_i$. Since $\tau = \tau^{-1}$, these edges are determined by the inclusions
$$[\epsilon_{-n} \cdots \underline{\epsilon_0} \cdots \epsilon_n] = \tau^{\pm 1}([\ol{\epsilon}_{-n} \cdots \underline{\ol{\epsilon}_0} \cdots \ol{\epsilon}_n]) \subseteq \tau^{\pm 1}([\ol{\epsilon}_{-n+1} \cdots \underline{\ol{\epsilon}_0} \cdots \ol{\epsilon}_{n-1}]).$$
Figure \ref{figure:bitwise_not} illustrates the first two levels of the associated $h$-diagram $(F,D)$.

\begin{figure}[H]
\begin{tikzpicture}
	\node[circle,fill=black,scale=0.4,label=above:\textcolor{black}{$X$}] (A1) at (0,4) {};
    \node[circle,fill=black,scale=0.4,label=above:\textcolor{black}{$[\underline{0}]$}] (B1) at (-4,2) {};
    \node[circle,fill=black,scale=0.4,label=above:\textcolor{black}{$[\underline{1}]$}] (B2) at (4,2) {};
	\node[circle,fill=black,scale=0.4,label=below:\textcolor{black}{$[0\underline{0}0]$}] (C1) at (-7,0) {};
	\node[circle,fill=black,scale=0.4,label=below:\textcolor{black}{$[1\underline{0}0]$}] (C2) at (-5,0) {};
	\node[circle,fill=black,scale=0.4,label=below:\textcolor{black}{$[0\underline{0}1]$}] (C3) at (-3,0) {};
	\node[circle,fill=black,scale=0.4,label=below:\textcolor{black}{$[1\underline{0}1]$}] (C4) at (-1,0) {};
	\node[circle,fill=black,scale=0.4,label=below:\textcolor{black}{$[0\underline{1}0]$}] (C5) at (1,0) {};
	\node[circle,fill=black,scale=0.4,label=below:\textcolor{black}{$[1\underline{1}0]$}] (C6) at (3,0) {};
	\node[circle,fill=black,scale=0.4,label=below:\textcolor{black}{$[0\underline{1}1]$}] (C7) at (5,0) {};
	\node[circle,fill=black,scale=0.4,label=below:\textcolor{black}{$[1\underline{1}1]$}] (C8) at (7,0) {};
	\draw[-,blue,postaction={on each segment={mid arrow=blue}}] (B1) to (A1);
	\draw[-,blue,postaction={on each segment={mid arrow=blue}}] (B2) to (A1);
	\draw[-,blue,postaction={on each segment={mid arrow=blue}}] (C1) to (B1);
	\draw[-,blue,postaction={on each segment={mid arrow=blue}}] (C2) to (B1);
	\draw[-,blue,postaction={on each segment={mid arrow=blue}}] (C3) to (B1);
	\draw[-,blue,postaction={on each segment={mid arrow=blue}}] (C4) to (B1);
	\draw[-,blue,postaction={on each segment={mid arrow=blue}}] (C5) to (B2);
	\draw[-,blue,postaction={on each segment={mid arrow=blue}}] (C6) to (B2);
	\draw[-,blue,postaction={on each segment={mid arrow=blue}}] (C7) to (B2);
	\draw[-,blue,postaction={on each segment={mid arrow=blue}}] (C8) to (B2);
	\draw[dashed,red,postaction={on each segment={mid arrow=red}}] (B1) to [out=90-10,in=180+10] (A1);
	\draw[dashed,red,postaction={on each segment={mid arrow=red}}] (B2) to [out=90+10,in=360-10] (A1);
	\draw[dashed,red,postaction={on each segment={mid arrow=red}}] (C1) to (B2);
	\draw[dashed,red,postaction={on each segment={mid arrow=red}}] (C2) to (B2);
	\draw[dashed,red,postaction={on each segment={mid arrow=red}}] (C3) to (B2);
	\draw[dashed,red,postaction={on each segment={mid arrow=red}}] (C4) to (B2);
	\draw[dashed,red,postaction={on each segment={mid arrow=red}}] (C5) to (B1);
	\draw[dashed,red,postaction={on each segment={mid arrow=red}}] (C6) to (B1);
	\draw[dashed,red,postaction={on each segment={mid arrow=red}}] (C7) to (B1);
	\draw[dashed,red,postaction={on each segment={mid arrow=red}}] (C8) to (B1);
	\node at (7.7,4){$\calP_0$};
	\node at (7.7,2){$\calP_1$};
	\node at (7.7,0){$\calP_2$};
	\draw[black,dotted] (-7,-0.8) -- (-7,0);
	\draw[black,dotted] (-5,-0.8) -- (-5,0);
	\draw[black,dotted] (-3,-0.8) -- (-3,0);
	\draw[black,dotted] (-1,-0.8) -- (-1,0);
	\draw[black,dotted] (1,-0.8) -- (1,0);
	\draw[black,dotted] (3,-0.8) -- (3,0);
	\draw[black,dotted] (5,-0.8) -- (5,0);
	\draw[black,dotted] (7,-0.8) -- (7,0);
\end{tikzpicture}
\caption{The first two levels of the $h$-diagram associated to the bit-wise NOT. The blue edges are solid, and the red edges are dashed.}
\label{figure:bitwise_not}
\end{figure}
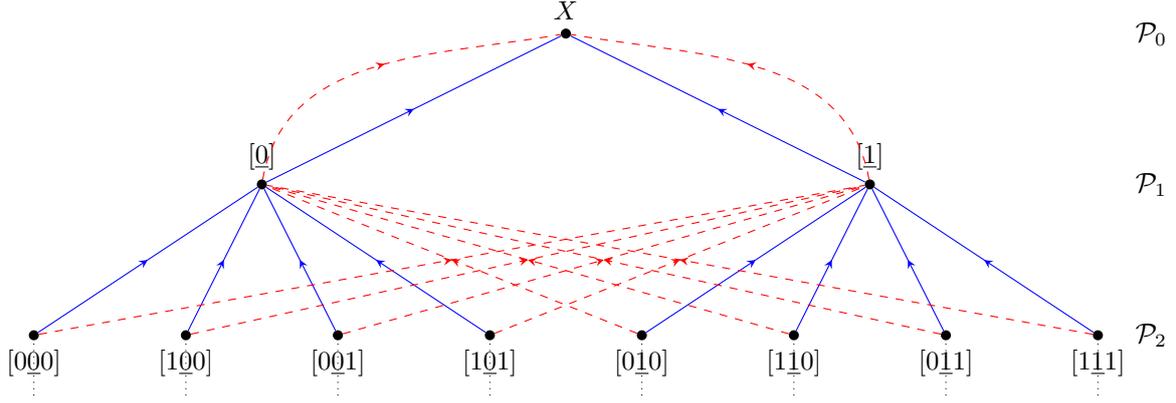

As with the fixed points in the previous example, the involution property $\tau^2 = \mathrm{id}_X$ is precisely detected by the diagram. We formalize and generalize how dynamical features can be read from the combinatorics of $(F,D)$ in the following section, Section \ref{section:dynamical_properties}, see in particular Theorem \ref{theorem:global_periodicity}.

\subsection{Example: the odometer map on the Cantor space}\label{example-odometer}

We now revisit the odometer map $(X,\operatorname{ad})$. Recall that $X$ is again the Cantor space, but we now use a different symbolic representation: the space of one-sided infinite sequences of zeros and ones. We use the $\operatorname{ad}$-refined sequence of partitions already considered in Section \ref{subsection-dynamical.systems.partitions}, namely $\calP_0 = \{X\}$ and $\calP_{n+1} = \{[\underline{\epsilon_0} \cdots \epsilon_n] \mid \epsilon_i \in \{0,1\} \text{ for all } i\}$ for $n \geq 0$. As previously mentioned, the map $\operatorname{ad}$ permutes the blocks within a fixed partition level $\calP_n$. More concretely, we have $\operatorname{ad}([\underline{0} \epsilon_1 \cdots \epsilon_n]) = [\underline{1} \epsilon_1 \cdots \epsilon_n]$, $\operatorname{ad}([\underline{1} \stackrel{l}{\cdots} 1 0\epsilon_{l+1} \cdots \epsilon_n]) = [\underline{0} \stackrel{l}{\cdots} 0 1 \epsilon_{l+1} \cdots \epsilon_n]$ for $l \geq 1$, and $\operatorname{ad}([\underline{1} \cdots 1]) = [\underline{0} \cdots 0]$.

We construct the $h$-diagram $(F,D)$ associated with the sequence $\{\calP_n\}_{n \geq 0}$. The zeroth level is identical to those in the previous two examples, so we fix $n \geq 1$ and focus on the $n$-th level $(F^{0,n} \sqcup F^{0,n+1},F^{1,n})$. Here $|F^{0,n}| = 2^n$ and $|F^{0,n+1}| = 2^{n+1}$. Thus, $F^{1,n}$ contains $2^{n+1}$ blue edges and $2^{n+1}$ red edges. The blue edges correspond to the inclusions:
$$[\underline{\epsilon_0} \cdots \epsilon_{n-1} \epsilon_n] \subseteq [\underline{\epsilon_0} \cdots \epsilon_{n-1}],$$
while the red edges correspond, for even $n$, to the inclusions:
$$[\underline{\epsilon_0} \cdots \epsilon_n] \subseteq \begin{cases} \operatorname{ad}([\underline{0} \epsilon_1 \cdots \epsilon_{n-1}]) & \text{if } \epsilon_0 = 1, \\ \operatorname{ad}([\underline{1} \stackrel{l}{\cdots} 1 0 \epsilon_{l+1} \cdots \epsilon_{n-1}]) & \text{if } \epsilon_i = 0 \text{ for all } i \in \{0,\dots,l-1\} \text{ and } \epsilon_l = 1, \\ \operatorname{ad}([\underline{1} \cdots 1]) & \text{if } \epsilon_i = 0 \text{ for all } i \in \{0,\dots,n\}. \end{cases}$$
and, for odd $n$, to the inclusions:
$$[\underline{\epsilon_0} \cdots \epsilon_n] \subseteq \begin{cases} \operatorname{ad}^{-1}([\underline{1} \epsilon_1 \cdots \epsilon_{n-1}]) & \text{if } \epsilon_0 = 0, \\ \operatorname{ad}^{-1}([\underline{0} \stackrel{l}{\cdots} 0 1 \epsilon_{l+1} \cdots \epsilon_{n-1}]) & \text{if } \epsilon_i = 1 \text{ for all } i \in \{0,\dots,l-1\} \text{ and } \epsilon_l = 0, \\ \operatorname{ad}^{-1}([\underline{0} \cdots 0]) & \text{if } \epsilon_i = 1 \text{ for all } i \in \{0,\dots,n\}. \end{cases}$$
For example, for $n=2$, the inclusion $[\underline{\epsilon_0} \epsilon_1 \epsilon_2] \subseteq [\underline{\epsilon_0} \epsilon_1]$ yields the blue edge $[\underline{\epsilon_0} \epsilon_1 \epsilon_2] \textcolor{blue}{\to} [\underline{\epsilon_0} \epsilon_1]$. Meanwhile, the inclusions $[\underline{1} \epsilon_1 \epsilon_2] \subseteq \operatorname{ad}([\underline{0} \epsilon_1])$, $[\underline{0} 1 \epsilon_2] \subseteq \operatorname{ad}([\underline{1} 0])$ and $[\underline{0} 0 \epsilon_2] \subseteq \operatorname{ad}([\underline{1} 1])$ dictate the red edges, which are drawn as $[\underline{1} \epsilon_1 \epsilon_2] \textcolor{red}{\to} [\underline{0} \epsilon_1]$, $[\underline{0} 1 \epsilon_2] \textcolor{red}{\to} [\underline{1} 0]$ and $[\underline{0} 0 \epsilon_2] \textcolor{red}{\to} [\underline{1} 1]$.

The resulting structure for the first three levels is depicted in Figure \ref{figure:odometer_Cantor}.

\begin{figure}[H]
\begin{tikzpicture}
	\node[circle,fill=black,scale=0.4,label=above:\textcolor{black}{$X$}] (A1) at (0,6) {};
    \node[circle,fill=black,scale=0.4,label=above:\textcolor{black}{$[\underline{0}]$}] (B1) at (-4,4) {};
    \node[circle,fill=black,scale=0.4,label=above:\textcolor{black}{$[\underline{1}]$}] (B2) at (4,4) {};
    \node[circle,fill=black,scale=0.4,label=above:\textcolor{black}{$[\underline{0}0]$}] (C1) at (-6,2) {};
    \node[circle,fill=black,scale=0.4,label=above:\textcolor{black}{$[\underline{0}1]$}] (C2) at (-2,2) {};
    \node[circle,fill=black,scale=0.4,label=above:\textcolor{black}{$[\underline{1}1]$}] (C3) at (2,2) {};
    \node[circle,fill=black,scale=0.4,label=above:\textcolor{black}{$[\underline{1}0]$}] (C4) at (6,2) {};
	\node[circle,fill=black,scale=0.4,label=below:\textcolor{black}{$[\underline{0}00]$}] (D1) at (-7,0) {};
	\node[circle,fill=black,scale=0.4,label=below:\textcolor{black}{$[\underline{0}01]$}] (D2) at (-5,0) {};
	\node[circle,fill=black,scale=0.4,label=below:\textcolor{black}{$[\underline{0}10]$}] (D3) at (-3,0) {};
	\node[circle,fill=black,scale=0.4,label=below:\textcolor{black}{$[\underline{0}11]$}] (D4) at (-1,0) {};
	\node[circle,fill=black,scale=0.4,label=below:\textcolor{black}{$[\underline{1}00]$}] (D5) at (1,0) {};
	\node[circle,fill=black,scale=0.4,label=below:\textcolor{black}{$[\underline{1}01]$}] (D6) at (3,0) {};
	\node[circle,fill=black,scale=0.4,label=below:\textcolor{black}{$[\underline{1}10]$}] (D7) at (5,0) {};
	\node[circle,fill=black,scale=0.4,label=below:\textcolor{black}{$[\underline{1}11]$}] (D8) at (7,0) {};
	\draw[-,blue,postaction={on each segment={mid arrow=blue}}] (B1) to (A1);
	\draw[-,blue,postaction={on each segment={mid arrow=blue}}] (B2) to (A1);
	\draw[-,blue,postaction={on each segment={mid arrow=blue}}] (C1) to (B1);
	\draw[-,blue,postaction={on each segment={mid arrow=blue}}] (C2) to (B1);
	\draw[-,blue,postaction={on each segment={mid arrow=blue}}] (C3) to (B2);
	\draw[-,blue,postaction={on each segment={mid arrow=blue}}] (C4) to (B2);
	\draw[-,blue,postaction={on each segment={mid arrow=blue}}] (D1) to (C1);
	\draw[-,blue,postaction={on each segment={mid arrow=blue}}] (D2) to (C1);
	\draw[-,blue,postaction={on each segment={mid arrow=blue}}] (D3) to (C2);
	\draw[-,blue,postaction={on each segment={mid arrow=blue}}] (D4) to (C2);
    \draw[-,blue,postaction={on each segment={mid arrow=blue}}] (D5) to (C3);
	\draw[-,blue,postaction={on each segment={mid arrow=blue}}] (D6) to (C3);
	\draw[-,blue,postaction={on each segment={mid arrow=blue}}] (D7) to (C4);
	\draw[-,blue,postaction={on each segment={mid arrow=blue}}] (D8) to (C4);
	\draw[dashed,red,postaction={on each segment={mid arrow=red}}] (B1) to [out=90-10,in=180+10] (A1);
	\draw[dashed,red,postaction={on each segment={mid arrow=red}}] (B2) to [out=90+10,in=360-10] (A1);
	\draw[dashed,red,postaction={on each segment={mid arrow=red}}] (C1) to (B2);
	\draw[dashed,red,postaction={on each segment={mid arrow=red}}] (C2) to (B2);
	\draw[dashed,red,postaction={on each segment={mid arrow=red}}] (C3) to (B1);
	\draw[dashed,red,postaction={on each segment={mid arrow=red}}] (C4) to (B1);
	\draw[dashed,red,postaction={on each segment={mid arrow=red}}] (D1) to (C3);
	\draw[dashed,red,postaction={on each segment={mid arrow=red}}] (D2) to (C3);
	\draw[dashed,red,postaction={on each segment={mid arrow=red}}] (D3) to (C4);
    \draw[dashed,red,postaction={on each segment={mid arrow=red}}] (D4) to (C4);
    \draw[dashed,red,postaction={on each segment={mid arrow=red}}] (D5) to (C1);
    \draw[dashed,red,postaction={on each segment={mid arrow=red}}] (D6) to (C1);
    \draw[dashed,red,postaction={on each segment={mid arrow=red}}] (D7) to (C2);
    \draw[dashed,red,postaction={on each segment={mid arrow=red}}] (D8) to (C2);
	\node at (7.7,6){$\calP_0$};
	\node at (7.7,4){$\calP_1$};
	\node at (7.7,2){$\calP_2$};
  	\node at (7.7,0){$\calP_3$};
	\draw[black,dotted] (-7,-0.8) -- (-7,0);
	\draw[black,dotted] (-5,-0.8) -- (-5,0);
	\draw[black,dotted] (-3,-0.8) -- (-3,0);
	\draw[black,dotted] (-1,-0.8) -- (-1,0);
	\draw[black,dotted] (1,-0.8) -- (1,0);
	\draw[black,dotted] (3,-0.8) -- (3,0);
	\draw[black,dotted] (5,-0.8) -- (5,0);
	\draw[black,dotted] (7,-0.8) -- (7,0);
\end{tikzpicture}
\caption{The first three levels of the $h$-diagram associated to the odometer map. The blue edges are solid, and the red edges are dashed.}
\label{figure:odometer_Cantor}
\end{figure}
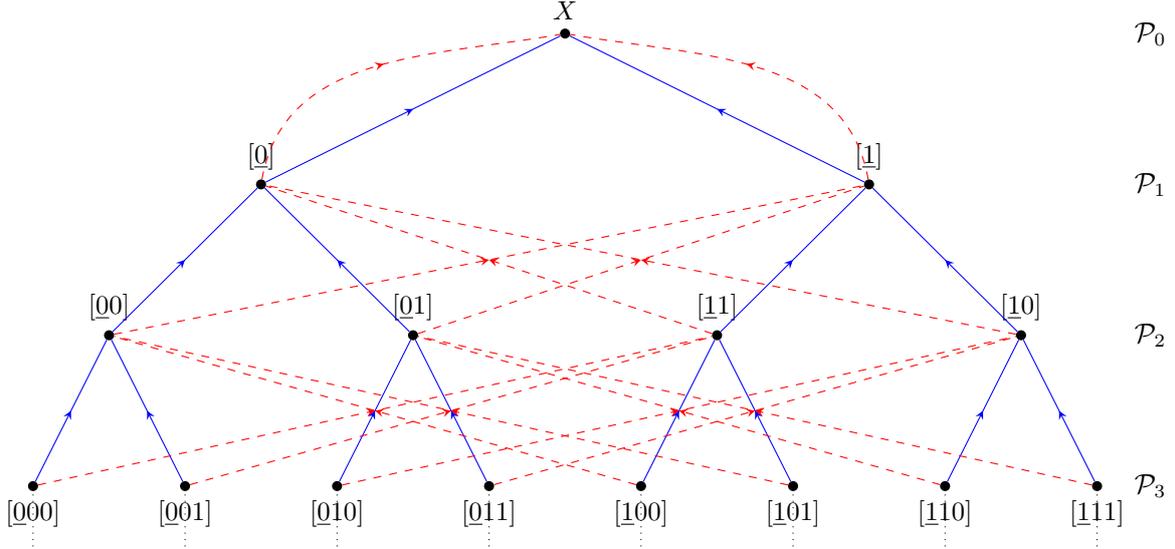

An important dynamical property of the odometer $(X,\operatorname{ad})$ is that it is a \textit{minimal} dynamical system, meaning that there are no proper, non-empty subsets $Y \subsetneq X$ that are closed and $\operatorname{ad}$-invariant. As before, this dynamical property can be characterized purely in terms of the combinatorics of the $h$-diagram $(F,D)$. We will pursue this objective in the upcoming, where we will characterize the broader concept of \textit{essentially minimal} dynamical systems. Before that, however, let us examine our final concrete example, which will serve as a prototype for essential minimality.

\subsection{Example: the shift map on the one-point compactification of the integers}\label{example-compactified.odometer}

Finally, let us consider the shift map $(X,\operatorname{add})$ defined on the totally disconnected compact space $X = \Z^{\ast}$, which is \textit{not} a Cantor space. The homeomorphism here is simply the addition map $\operatorname{add} \colon X \to X$ defined by $\operatorname{add}(n) = n+1$ and $\operatorname{add}(\infty) = \infty$. Recall that the $\operatorname{add}$-refined sequence of partitions we considered in Section \ref{subsection-dynamical.systems.partitions} is given by $\calP_0 = \{X\}$ and $\calP_{n+1} = \{\{j\} \mid j \in \{-n,\dots,n\}\} \cup \{ V_n\}$ for $n \geq 0$. We use this sequence to construct the associated $h$-diagram $(F,D)$.

Set $F^{0,n} := \calP_n$, so that $|F^{0,0}| = 1$ and $|F^{0,n}| = 2n$ for $n \geq 1$. As with our previous examples, the zeroth level remains unchanged, so we turn our attention to the $n$-th level $(F^{0,n} \sqcup F^{0,n+1},F^{1,n})$ for $n \geq 1$. Here $|F^{0,n}| = 2n$ and $|F^{0,n+1}| = 2n+2$. Thus, $F^{1,n}$ contains $2n+2$ blue edges and $2n+2$ red edges. The blue edges correspond to the trivial inclusions $\{j\} \subseteq \{j\}$ for $j \in \{-n+1,\dots,n-1\}$ and $\{-n\}, \, \{n\}, \, V_n \subseteq V_{n-1}$. 

The red edges correspond to $\{j\} \subseteq \operatorname{add}\left(\{j-1\}\right)$ for $j \in \{-n+2,\dots,n\}$ and $\{-n\}, \, \{-n+1\}, \, V_n \subseteq \operatorname{add} \left( V_{n-1} \right)$ if $n$ is even, and to $\{j\} \subseteq \operatorname{add}^{-1}\left(\{j+1\}\right)$ for $j \in \{-n,\dots,n-2\}$ and $\{n-1\}, \, \{n\}, \, V_n \subseteq \operatorname{add}^{-1} \left( V_{n-1} \right)$ if $n$ is odd.

To illustrate, take $n=1$. The inclusions $\{0\} \subseteq \{0\}$ and $\{-1\}, \, \{1\}, \, V_1 \subseteq V_0$ dictate the targets of the blue arrows originating from $\calP_1$. These are drawn as $\{0\} \textcolor{blue}{\to} \{0\}$, $\{-1\} \textcolor{blue}{\to} V_0$, $\{1\} \textcolor{blue}{\to} V_0$ and $V_1 \textcolor{blue}{\to} V_0$. Simultaneously, the inclusions $\{-1\} \subseteq \operatorname{add}^{-1}(\{0\})$ alongside $\{0\}, \, \{1\}, \, V_1 \subseteq \operatorname{add}^{-1}\left(V_0\right)$ determine the routing of the corresponding red arrows. Thus, $\{-1\} \textcolor{red}{\to} \{0\}$, $\{0\} \textcolor{red}{\to} V_0$, $\{1\} \textcolor{red}{\to} V_0$ and $V_1 \textcolor{red}{\to} V_0$.

The complete structure generated by these rules over the first four levels is visualized in Figure \ref{figure-compactified.odometer}.

\begin{figure}[H]
\begin{tikzpicture}
	\node[circle,fill=black,scale=0.4,label=above:\textcolor{black}{$X$}] (A1) at (-7,6) {};
    \node[circle,fill=black,scale=0.4,label=above:\textcolor{black}{$\{0\}$}] (B1) at (-7,4.5) {};
    \node[circle,fill=black,scale=0.4,label=above:\textcolor{black}{$V_0$}] (B2) at (-4,4.5) {};
    \node[circle,fill=black,scale=0.4,label=above:\textcolor{black}{$\{0\}$}] (C1) at (-7,3) {};
    \node[circle,fill=black,scale=0.4,label=above:\textcolor{black}{$\{-1\}$}] (C2) at (-5,3) {};
    \node[circle,fill=black,scale=0.4,label=above:\textcolor{black}{$\{1\}$}] (C3) at (-3,3) {};
    \node[circle,fill=black,scale=0.4,label=above:\textcolor{black}{$V_1$}] (C4) at (0,3) {};
    \node[circle,fill=black,scale=0.4,label=above:\textcolor{black}{$\{0\}$}] (D1) at (-7,1.5) {};
    \node[circle,fill=black,scale=0.4,label=above:\textcolor{black}{$\{-1\}$}] (D2) at (-5,1.5) {};
    \node[circle,fill=black,scale=0.4,label=above:\textcolor{black}{$\{1\}$}] (D3) at (-3,1.5) {};
    \node[circle,fill=black,scale=0.4,label=above:\textcolor{black}{$\{-2\}$}] (D4) at (-1,1.5) {};
    \node[circle,fill=black,scale=0.4,label=above:\textcolor{black}{$\{2\}$}] (D5) at (1,1.5) {};
    \node[circle,fill=black,scale=0.4,label=above:\textcolor{black}{$V_2$}] (D6) at (4,1.5) {};
	\node[circle,fill=black,scale=0.4,label=below:\textcolor{black}{$\{0\}$}] (E1) at (-7,0) {};
	\node[circle,fill=black,scale=0.4,label=below:\textcolor{black}{$\{-1\}$}] (E2) at (-5,0) {};
	\node[circle,fill=black,scale=0.4,label=below:\textcolor{black}{$\{1\}$}] (E3) at (-3,0) {};
	\node[circle,fill=black,scale=0.4,label=below:\textcolor{black}{$\{-2\}$}] (E4) at (-1,0) {};
	\node[circle,fill=black,scale=0.4,label=below:\textcolor{black}{$\{2\}$}] (E5) at (1,0) {};
	\node[circle,fill=black,scale=0.4,label=below:\textcolor{black}{$\{-3\}$}] (E6) at (3,0) {};
	\node[circle,fill=black,scale=0.4,label=below:\textcolor{black}{$\{3\}$}] (E7) at (5,0) {};
	\node[circle,fill=black,scale=0.4,label=below:\textcolor{black}{$V_3$}] (E8) at (7,0) {};
	\draw[-,blue,postaction={on each segment={mid arrow=blue}}] (B1) to (A1);
	\draw[-,very thick,blue,postaction={on each segment={mid arrow=blue}}] (B2) to (A1);
	\draw[-,blue,postaction={on each segment={mid arrow=blue}}] (C1) to (B1);
	\draw[-,blue,postaction={on each segment={mid arrow=blue}}] (C2) to (B2);
	\draw[-,blue,postaction={on each segment={mid arrow=blue}}] (C3) to (B2);
	\draw[-,very thick,blue,postaction={on each segment={mid arrow=blue}}] (C4) to (B2);
	\draw[-,blue,postaction={on each segment={mid arrow=blue}}] (D1) to (C1);
	\draw[-,blue,postaction={on each segment={mid arrow=blue}}] (D2) to (C2);
	\draw[-,blue,postaction={on each segment={mid arrow=blue}}] (D3) to (C3);
	\draw[-,blue,postaction={on each segment={mid arrow=blue}}] (D4) to (C4);
    \draw[-,blue,postaction={on each segment={mid arrow=blue}}] (D5) to (C4);
	\draw[-,very thick,blue,postaction={on each segment={mid arrow=blue}}] (D6) to (C4);
	\draw[-,blue,postaction={on each segment={mid arrow=blue}}] (E1) to (D1);
	\draw[-,blue,postaction={on each segment={mid arrow=blue}}] (E2) to (D2);
    \draw[-,blue,postaction={on each segment={mid arrow=blue}}] (E3) to (D3);
	\draw[-,blue,postaction={on each segment={mid arrow=blue}}] (E4) to (D4);
    \draw[-,blue,postaction={on each segment={mid arrow=blue}}] (E5) to (D5);
	\draw[-,blue,postaction={on each segment={mid arrow=blue}}] (E6) to (D6);
    \draw[-,blue,postaction={on each segment={mid arrow=blue}}] (E7) to (D6);
	\draw[-,very thick,blue,postaction={on each segment={mid arrow=blue}}] (E8) to (D6);
	\draw[dashed,red,postaction={on each segment={mid arrow=red}}] (B1) to [out=90+45,in=180+45] (A1);
	\draw[dashed,red,postaction={on each segment={mid arrow=red}}] (B2) to [out=90+10,in=360-10] (A1);
	\draw[dashed,red,postaction={on each segment={mid arrow=red}}] (C1) to (B2);
	\draw[dashed,red,postaction={on each segment={mid arrow=red}}] (C2) to (B1);
	\draw[dashed,red,postaction={on each segment={mid arrow=red}}] (C3) to [out=90-20,in=360-45] (B2);
	\draw[dashed,red,postaction={on each segment={mid arrow=red}}] (C4) to [out=90+10,in=360-10] (B2);
	\draw[dashed,red,postaction={on each segment={mid arrow=red}}] (D1) to (C2);
	\draw[dashed,red,postaction={on each segment={mid arrow=red}}] (D2) to (C4);
	\draw[dashed,red,postaction={on each segment={mid arrow=red}}] (D3) to (C1);
    \draw[dashed,red,postaction={on each segment={mid arrow=red}}] (D4) to [out=90+20,in=180+45] (C4);
    \draw[dashed,red,postaction={on each segment={mid arrow=red}}] (D5) to (C3);
    \draw[dashed,red,postaction={on each segment={mid arrow=red}}] (D6) to [out=90+10,in=360-10] (C4);
    \draw[dashed,red,postaction={on each segment={mid arrow=red}}] (E1) to (D3);
    \draw[dashed,red,postaction={on each segment={mid arrow=red}}] (E2) to (D1);
    \draw[dashed,red,postaction={on each segment={mid arrow=red}}] (E3) to (D5);
    \draw[dashed,red,postaction={on each segment={mid arrow=red}}] (E4) to (D2);
    \draw[dashed,red,postaction={on each segment={mid arrow=red}}] (E5) to (D6);
    \draw[dashed,red,postaction={on each segment={mid arrow=red}}] (E6) to (D4);
    \draw[dashed,red,postaction={on each segment={mid arrow=red}}] (E7) to [out=90-20,in=360-45] (D6);
    \draw[dashed,red,postaction={on each segment={mid arrow=red}}] (E8) to [out=90+10,in=360-10] (D6);
	\node at (-7.7,6){$\calP_0$};
	\node at (-7.7,4.5){$\calP_1$};
	\node at (-7.7,3){$\calP_2$};
  	\node at (-7.7,1.5){$\calP_3$};
    \node at (-7.7,0){$\calP_4$};
	\draw[black,dotted] (-7,-0.8) -- (-7,0);
	\draw[black,dotted] (-5,-0.8) -- (-5,0);
	\draw[black,dotted] (-3,-0.8) -- (-3,0);
	\draw[black,dotted] (-1,-0.8) -- (-1,0);
	\draw[black,dotted] (1,-0.8) -- (1,0);
	\draw[black,dotted] (3,-0.8) -- (3,0);
	\draw[black,dotted] (5,-0.8) -- (5,0);
	\draw[black,dotted] (7,-0.8) -- (7,0);
\end{tikzpicture}
\caption{The first four levels of the $h$-diagram associated to the shift map on $\Z^{\ast}$. The blue edges are solid, and the red edges are dashed.}
\label{figure-compactified.odometer}
\end{figure}
Note that from Figure \ref{figure-compactified.odometer}, we can readily infer the presence of a fixed point for $\operatorname{add}$, which is $\bigcap_{n \geq 0} V_n = \{\infty\}$, as expected.

As mentioned previously, this serves as a prototypical example of an essentially minimal dynamical system, a concept we will formally characterize in the next section.

\section{A sampling of dynamical properties read from the \texorpdfstring{$h$}{}-diagram}\label{section:dynamical_properties}

In the previous section, we detailed the construction of the $h$-diagram $(F,D)$ associated with a dynamical system $(X,h)$. Recall that the definition of the red edges is strictly governed by the homeomorphism $h$. This naturally raises a fundamental question: to what extent can the dynamical properties of $(X,h)$ be strictly ``read off'' from the purely combinatorial data of its $h$-diagram?

In this section, we demonstrate that the $h$-diagram is indeed a powerful tool that encodes several key dynamical features of the system. We provide a sampling of such properties, beginning with global periodicity.

\subsection{Globally periodic homeomorphisms}

The simplest type of dynamical system is one in which the homeomorphism $h$ has finite order, meaning that $h^m = \mathrm{id}_X$ for some integer $m \geq 1$. When this occurs, the homeomorphism $h$ is said to be \textit{globally periodic} with period $m$. In such a system, the action of $\Z$ on $X$ induced by $h$ descends to an action of the finite group $\Z_m$ on $X$. Consequently, one should expect this global periodicity to severely restrict the structure of the associated $h$-diagram. In what follows, we characterize exactly how the condition $h^m = \mathrm{id}_X$ manifests in $(F,D)$. We begin with a definition.

\begin{definition}[Global periodicity for $h$-diagrams]\label{definition:global_periodicity}
Fix $m \in \N$. An $h$-diagram $(F,D)$ has \textit{global periodicity $m$} if for each path $\mu \in \operatorname{Path}_{\text{fin}}(F^1)$ consisting of alternating red and blue edges (i.e., consecutive edges in $\mu$ have different colors) and containing exactly $m$ red edges, there exists a path $\mu' \in \operatorname{Path}_{\text{fin}}(F^1)$ consisting entirely of blue edges such that
$$r(\mu) = r(\mu') \quad \text{ and } \quad s(\mu) = s(\mu').$$
\end{definition}

Pictorially, global periodicity $m$ states that, for any pair of vertices $v,w \in F^0$, if it is possible to go from $w$ to $v$ by following a path of alternating red and blue edges having exactly $m$ red edges, then it is also possible to go from $w$ to $v$ by another path consisting entirely of blue edges. This is reflected in Figure \ref{figure-property.Zm}.

Note that an alternating path containing exactly $m$ red edges must have length $2m-1$, $2m$, or $2m+1$, depending on whether it starts and/or ends with a blue edge. However, to establish global periodicity $m$, it suffices to verify the condition for paths of length exactly $2m-1$ (those beginning and ending with a red edge). Indeed, if an alternating path $\mu$ begins or ends with a blue edge, one can simply apply the property to the maximal subpath that starts and ends with red edges, and then absorb the extremal blue edges of $\mu$ into $\mu'$ to satisfy the required property.

\begin{figure}[H]
\centering
\begin{tikzpicture}
	\path [draw=blue,postaction={on each segment={mid arrow=blue}}]
		(-1.5,-4) -- node[left,blue]{$e_2$} (-1,-3)
		(-2,-2) -- node[above,blue]{$e_1$} (-0.5,-1)
		;
    \draw[dotted,blue,postaction={on each segment={mid arrow=blue}}] (-1,-5) -- node[above,blue]{$e'_4$} (1.5,-4);
    \draw[dotted,blue,postaction={on each segment={mid arrow=blue}}] (1.5,-4) -- node[left,blue]{$e'_3$} (1,-3);
    \draw[dotted,blue,postaction={on each segment={mid arrow=blue}}] (1,-3) -- node[left,blue]{$e'_2$} (1.5,-2);
    \draw[dotted,blue,postaction={on each segment={mid arrow=blue}}] (1.5,-2) -- node[left,blue]{$e'_1$} (0.5,-1);
    \draw[dotted,blue,postaction={on each segment={mid arrow=blue}}] (0.5,-1) -- node[right,blue]{$e'_0$} (0,0);
    \path [dashed,draw=red,postaction={on each segment={mid arrow=red}}]
		(-1,-5) -- node[left,red]{$f_2$} (-1.5,-4)
		(-1,-3) -- node[above,red]{$f_1$} (-2,-2)
        (-0.5,-1) -- node[left,red]{$f_0$} (0,0)
		;
	\node[circle,fill=black,scale=0.4,label=above:$v$] (T) at (0,0) {};
    \node[circle,fill=black,scale=0.4] (T) at (-0.5,-1) {};
    \node[circle,fill=black,scale=0.4] (T) at (0.5,-1) {};
    \node[circle,fill=black,scale=0.4] (T) at (-2,-2) {};
    \node[circle,fill=black,scale=0.4] (T) at (1.5,-2) {};
    \node[circle,fill=black,scale=0.4] (T) at (-1,-3) {};
    \node[circle,fill=black,scale=0.4] (T) at (1,-3) {};
    \node[circle,fill=black,scale=0.4] (T) at (-1.5,-4) {};
    \node[circle,fill=black,scale=0.4] (T) at (1.5,-4) {};
    \node[circle,fill=black,scale=0.4,label=below:$w$] (T) at (-1,-5) {};
\end{tikzpicture}
\caption{Visualization of \textit{global periodicity $m$} for $m = 3$. The solid and dotted edges are blue, while the dashed edges are red.}
\label{figure-property.Zm}
\end{figure}
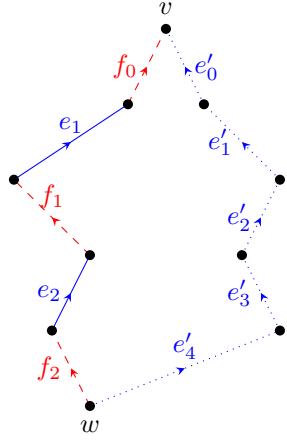

We are now ready to relate global periodicity of $h$ with global periodicity of the associated $h$-diagram.

\begin{theorem}\label{theorem:global_periodicity}
Let $X$ be a totally disconnected compact metric space and $h$ a homeomorphism on $X$. Let $\{\calP_n\}_{n \geq 0}$ be an $h$-refined sequence of partitions of $X$, and let $(F,D)$ denote the associated $h$-diagram. Then, for any $m \in \N$, the following are equivalent:
\begin{enumerate}
\item[(1)] $h$ is globally periodic with period $m$.
\item[(2)] $(F,D)$ has global periodicity $m$.
\end{enumerate}
\end{theorem}
\begin{proof}
Suppose first that $h$ is globally periodic with period $m$, so that $h^m = \mathrm{id}_X$, As remarked before, it suffices to consider a $(2m-1)$-path $\mu = f_0 \, e_1 \, f_1 \, \cdots \, e_{m-1} \, f_{m-1}$ consisting of alternating blue and red edges, starting and ending with a red edge.

For each $i \in \{0,\dots,m-1\}$, define $Z_{2i} := r(f_i)$ and $Z_{2i+1} := s(f_i)$. Thus $f_i = f(Z_{2i+1},Z_{2i})$, and moreover $e_i = e(Z_{2i},Z_{2i-1})$. Each $Z_j$ corresponds to a clopen set of $X$, and there exists $n \geq 0$ such that $Z_j \in \calP_{n+j}$ for all $j \in \{0,\dots,2m-1\}$.

We proceed by assuming that $n$ is even; the case where $n$ is odd is analogous (noting that $h^{-m} = \text{id}_X$) and is thus omitted. By construction, we have the relations $Z_{2i+1} \subseteq h(Z_{2i})$ and $Z_{2i} \subseteq Z_{2i-1}$ for all $i \in \{1,\dots,m-1\}$, together with the relation $Z_1 \subseteq h(Z_0)$. Thus,
$$Z_{2m-1} \subseteq h(Z_{2m-3}) \subseteq h^2(Z_{2m-5}) \subseteq \cdots \subseteq h^{m-1}(Z_1) \subseteq h^m(Z_0).$$
Since $h^m = \mathrm{id}_X$ by hypothesis, we conclude that $Z_{2m-1} \subseteq Z_0$. Because the partition sequence is refining, there exists a unique sequence of clopen sets $\tilde{Z}_j \in \calP_{n+j}$, $j \in \{1,\dots,2m-2\}$, such that
$$Z_{2m-1} \subseteq \tilde{Z}_{2m-2} \subseteq \cdots \subseteq \tilde{Z}_1 \subseteq Z_0.$$
This sequence corresponds to a $(2m-1)$-path $\mu' = e'_0 \, e'_1 \, \cdots \, e'_{2m-2}$ consisting of blue edges, where $e'_0 = e(\tilde{Z}_1,Z_0)$, $e'_{2m-2} = e(Z_{2m-1},\tilde{Z}_{2m-2})$ and $e'_j = e(\tilde{Z}_{j+1},\tilde{Z}_j)$ for $j \in \{1,\dots,2m-3\}$. Note that $r(\mu') = Z_0 = r(\mu)$ and $s(\mu') = Z_{2m-1} = s(\mu)$. We therefore conclude that $(F,D)$ has global periodicity $m$.

Conversely, suppose that $(F,D)$ has global periodicity $m$. Fix $x \in X$. For each $n \geq 0$, let $Z_n(x)$ be the unique clopen set in $\calP_n$ containing $x$. We have $Z_{n+1}(x) \subseteq Z_n(x)$ and $\bigcap_{n \geq 0} Z_n(x) = \{x\}$.

Fix $N \geq 2m-1$ an odd integer. Since each $\calP_{n+1}$ is finer than $\calP_n \wedge h(\calP_n)$, there exists a sequence of clopen sets $\tilde{Z}_j \in \calP_j$ for $j \in \{N-2m+1,\dots,N-2,N-1\}$, satisfying the relations $Z_N(x) \subseteq h(\tilde{Z}_{N-1})$ and $\tilde{Z}_{N-2i+1} \subseteq \tilde{Z}_{N-2i} \subseteq h(\tilde{Z}_{N-2i-1})$ for all $i \in \{1,\dots,m-1\}$. In particular,
$$Z_N(x) \subseteq h(\tilde{Z}_{N-1}) \subseteq h^2(\tilde{Z}_{N-3}) \subseteq \cdots \subseteq h^{m-1}(\tilde{Z}_{N-2m+3}) \subseteq h^m(\tilde{Z}_{N-2m+1}),$$
and hence $h^{-m}(x) \in \tilde{Z}_{N-2m+1}$.

Let $f_{m-1} = f(Z_N(x),\tilde{Z}_{N-1})$ and, more generally, let $e_{m-i} = e(\tilde{Z}_{N-2i+1},\tilde{Z}_{N-2i})$ (resp. $f_{m-i-1} = f(\tilde{Z}_{N-2i},\tilde{Z}_{N-2i-1})$) for $i \in \{1,\dots,m-1\}$. Note that $r(f_{m-i}) = s(e_{m-i})$ and $r(e_{m-i}) = s(f_{m-i-1})$ for all $i \in \{1,\dots,m-1\}$, so $\mu := f_0 \, e_1 \, f_1 \, \cdots \, e_{m-1} \, f_{m-1}$ is a $(2m-1)$-path consisting of alternating red and blue edges. Hence, we can apply global periodicity $m$ to $\mu$ to conclude that there exists a $(2m-1)$-path $\mu' = e'_0 \, e'_1 \, \cdots \, e'_{2m-2}$ consisting of blue edges, with $r(\mu) = \tilde{Z}_{N-2m+1} = r(\mu')$ and $s(\mu) = Z_N(x) = s(\mu')$.

Since $\mu'$ is a purely blue path, we deduce that $Z_N(x) \subseteq \tilde{Z}_{N-2m+1}$. Because the partition sequence is refining, $\tilde{Z}_{N-2m+1}$ must be the unique element of $\calP_{N-2m+1}$ containing $x$, so $\tilde{Z}_{N-2m+1} = Z_{N-2m+1}(x)$. Consequently,
$$h^{-m}(x) \in \tilde{Z}_{N-2m+1} = Z_{N-2m+1}(x).$$
Since this holds for all odd $N \geq 2m-1$, we obtain:
$$h^{-m}(x) \in \bigcap_{\substack{N \geq 2m-1 \\ N \text{ odd}}} Z_{N-2m+1}(x) = \bigcap_{n \geq 0} Z_n(x) = \{x\},$$
yielding $h^{-m}(x) = x$, and equivalently, $h^m(x) = x$. Since $x \in X$ was arbitrary, we conclude that $h^m = \mathrm{id}_X$.
\end{proof}

Thus, Theorem \ref{theorem:global_periodicity} establishes that the global periodicity of the system is faithfully captured by the combinatorics of its associated $h$-diagram. We next explore how similar combinatorial features can be used to detect (essential) minimality.

\subsection{Essentially minimal dynamical systems}\label{subsection:ess_min}

In this section, we characterize the class of $h$-diagrams that give rise to essentially minimal dynamical systems, emphasizing the special and important case of minimal systems.

\begin{definition}\label{definition:ess_min}\cite[Definition 1.2]{HPS1992}
Let $(X,h)$ be a dynamical system where $X$ is a totally disconnected compact metric space and $h \colon X \to X$ is a homeomorphism. A set $Z \subseteq X$ is called \textit{minimal} if it is closed, non-empty, $h$-invariant (i.e., $h(Z) = Z$), and it is minimal among such sets --- meaning that if $Z' \subseteq Z$ is another closed, non-empty, $h$-invariant subset of $X$, then $Z' = Z$.

We say that the pair $(X,h)$ is \textit{essentially minimal} if there is a unique minimal set, say $Y$. If $Y = X$, we say that $(X,h)$ is \textit{minimal}.
\end{definition}

By the hypothesis on $X$, a standard application of Zorn's Lemma shows that minimal sets always exist. Consequently, every closed, non-empty, $h$-invariant subset of $X$ contains a minimal set. In particular, if $(X,h)$ is essentially minimal with unique minimal set $Y$, then $Y$ is contained in every closed, non-empty, $h$-invariant subset of $X$.

We now introduce a graph-theoretic property for $h$-diagrams that mirrors essential minimality.

\begin{definition}
Let $(F,D)$ be a separated Bratteli diagram. Let $v,w \in F^0$. We say that $w$ is \textit{connected} to $v$, denoted by $w \leadsto v$, if there exists a path $\mu \in \operatorname{Path}_{\text{fin}}(F^1)$ such that $s(\mu) = w$ and $r(\mu) = v$. 
\end{definition}
Note that if $w \leadsto v$, it is necessarily true that $w \in F^{0,n}$ and $v \in F^{0,m}$ for some $0 \leq m \leq n$. We also admit paths of zero length, meaning $v \leadsto v$ for any vertex $v \in F^0$.

\begin{definition}[Essential minimality for diagrams]\label{definition:property_em}
Let $(F,D)$ be a separated Bratteli diagram. An infinite blue path $\gamma = e_0 \, e_1 \, \cdots \in \operatorname{Path}_{\infty}(F^1)$ satisfies \textit{property (EM)} if for every $i \geq 0$, there exists $n_i \in \N_0$ such that every vertex $v \in F^{0,n_i}$ is connected to $r(e_i)$.

We say that $(F,D)$ is \textit{essentially minimal} if there exists an infinite blue path $\gamma \in \operatorname{Path}_{\infty}(F^1)$ satisfying property (EM).
\end{definition}

Roughly speaking, property (EM) asserts the existence of an infinite blue path such that, for any vertex on this path, there is always some level in the graph $(F,D)$ from which every vertex is connected to the chosen vertex.

We now state the main theorem of this subsection. The remainder of the subsection is devoted to proving this result and deriving several corollaries.

\begin{theorem}\label{theorem:ess_min}
Let $X$ be a totally disconnected compact metric space and $h$ a homeomorphism on $X$. Let $\{\calP_n\}_{n \geq 0}$ be an $h$-refined sequence of partitions of $X$, and let $(F,D)$ denote the associated $h$-diagram. The following are equivalent:
\begin{enumerate}
\item[(1)] $(X,h)$ is essentially minimal.
\item[(2)] $(F,D)$ is essentially minimal.
\end{enumerate}
\end{theorem}

Before starting the proof of Theorem \ref{theorem:ess_min}, we need several preliminary results.

\begin{lemma}\label{lemma-connected.translates}
Let $Z_2 \in \calP_{m+2i}$ and $Z_1 \in \calP_m$ for some $m,i \geq 0$. The following are equivalent.
\begin{enumerate}
\item[(i)] $Z_2 \leadsto Z_1$, regarded as vertices in $F$.
\item[(ii)] $Z_2 \cap h^j(Z_1) \neq \emptyset$ for some $j \in \Z$ with $|j| \leq i$.
\item[(iii)] $Z_2 \subseteq h^j(Z_1)$ for some $j \in \Z$ with $|j| \leq i$.
\end{enumerate}
\end{lemma}
\begin{proof}
We start with $(i) \implies (iii)$. Take a finite path $\mu \in \operatorname{Path}_{\text{fin}}(F^1)$ for which $s(\mu) = Z_2$ and $r(\mu) = Z_1$, and note that $|\mu| = 2i$. We prove the result by strong induction on $i$.

If $i = 0$ then $Z_2 = \mu = Z_1$ and the result is obvious. For $i = 1$, we have $\mu = g_0 \, g_1$ for some edges $g_k \in F^1$. Let $Z' := r(g_1) = s(g_0)$. We distinguish several cases.

If both $g_0$ and $g_1$ are blue edges, then $g_0 = e(Z',Z_1)$ and $g_1 = e(Z_2,Z')$, yielding
$$Z_2 \subseteq Z' \subseteq Z_1 = h^0(Z_1).$$
Thus, the result holds. If $g_0$ is a blue edge and $g_1$ is a red edge, then $g_0 = e(Z',Z_1)$ and $g_1 = f(Z_2,Z')$, so that
$$Z_2 \subseteq h^{\pm 1}(Z') \subseteq h^{\pm 1}(Z_1)$$
depending on the parity of the level $m$. In either case, the result holds. The case where $g_0$ is red and $g_1$ is blue is symmetric to the previous one, so we omit it. Finally, if both $g_0$ and $g_1$ are red edges, then by Observation \ref{observations-1}(2), there exists a unique $2$-path consisting of blue edges with the same source $Z_2$ and range $Z_1$. We are back to the first case, for which the result is true.

Suppose now that the result holds true for all $i \in \{0,1,\dots,i_0\}$ for some $i_0 \geq 1$, and we prove it for $i = i_0+1$. Thus $\mu$ now has length $2i_0+2$. Write $\mu = g_0 \, g_1 \, \cdots \, g_{2i_0+1}$ with $g_k \in F^1$. Let $Z' := s(g_1)$. Then $Z_2$ is connected to $Z'$ by a path of length $2i_0$, and $Z'$ is connected to $Z_1$ by a path of length $2$. By strong induction, there are $j_1,j_2 \in \Z$, with $|j_1| \leq 1$ and $|j_2| \leq i_0$, such that $Z_2 \subseteq h^{j_2}(Z')$ and $Z' \subseteq h^{j_1}(Z_1)$. Hence
$$Z_2 \subseteq h^{j_2}\left(h^{j_1}(Z_1)\right) = h^{j_1+j_2}(Z_1)$$
with $|j_1+j_2| \leq i_0+1$, establishing the result.

The implication $(iii) \implies (ii)$ is trivial, so we will finish by showing $(ii) \implies (i)$. For the sake of brevity, we only prove the case where $m$ is even, leaving the odd case to the reader.

If $j=0$, then $Z_2 \cap Z_1 \neq \emptyset$. By the refining properties of the sequence $\{\calP_n\}_{n \geq 0}$, there exists a unique sequence of clopen sets $Z^{(k)} \in \calP_{m+k}$ for $k \in \{0,1,\dots,2i-1\}$ such that
$$Z_2 \subseteq Z^{(2i-1)} \subseteq Z^{(2i-2)} \subseteq \cdots \subseteq Z^{(1)} \subseteq Z^{(0)}.$$
Thus, $Z_2 \subseteq Z^{(0)}$. Since $Z_2 \cap Z_1 \neq \emptyset$, we deduce that $Z_1 \cap Z^{(0)} \neq \emptyset$. Because $Z_1$ and $Z^{(0)}$ both belong to $\calP_m$, uniqueness implies $Z^{(0)} = Z_1$. Therefore, $Z_2 \subseteq Z_1$, meaning $Z_2$ is connected to $Z_1$ via the path
$$\mu = e(Z^{(1)},Z_1) \, e(Z^{(2)},Z^{(1)}) \, \cdots \, e(Z_2,Z^{(2i-1)}).$$
If $j=1$, we have $Z_2 \cap h(Z_1) \neq \emptyset$. As argued above, there exist unique clopen sets $Z^{(k)} \in \calP_{m+k}$ for $k \in \{0,1,\dots,2i-1\}$ such that
$$Z_2 \subseteq Z^{(2i-1)}, \quad Z^{(2i-1)} \subseteq h(Z^{(2i-2)}), \quad \text{ and } \quad Z^{(2i-2)} \subseteq Z^{(2i-3)} \subseteq \cdots \subseteq Z^{(1)} \subseteq Z^{(0)}.$$
In particular, $Z_2 \subseteq h(Z^{(0)})$. Because $Z_2 \cap h(Z_1) \neq \emptyset$, it follows that $Z_1 \cap Z^{(0)} \neq \emptyset$. Thus, $Z^{(0)} = Z_1$, so $Z_2 \subseteq h(Z_1)$. In this case, a path connecting $Z_2$ to $Z_1$ is
$$\mu = e(Z^{(1)},Z_1) \, e(Z^{(2)},Z^{(1)}) \, \cdots \, e(Z^{(2i-2)},Z^{(2i-3)}) \, f(Z^{(2i-1)},Z^{(2i-2)}) \, e(Z_2,Z^{(2i-1)}).$$
The case $j=-1$ is similar: assuming $Z_2 \cap h^{-1}(Z_1) \neq \emptyset$, we find unique clopen sets $Z^{(k)} \in \calP_{m+j}$ for $k \in \{0,1,\dots,2i-1\}$ such that
$$Z_2 \subseteq h^{-1}(Z^{(2i-1)}), \quad \text{ and } \quad Z^{(2i-1)} \subseteq Z^{(2i-2)} \subseteq \cdots \subseteq Z^{(1)} \subseteq Z^{(0)}.$$
This yields $Z_2 \subseteq h^{-1}(Z^{(0)})$, forcing $Z^{(0)} = Z_1$ as before. Therefore, $Z_2$ is connected to $Z_1$ via the path 
$$\mu = e(Z^{(1)},Z_1) \, e(Z^{(2)},Z^{(1)}) \, \cdots \, e(Z^{(2i-1)},Z^{(2i-2)}) \, f(Z_2,Z^{(2i-1)}).$$
This logic extends to any $j \in \Z$ with $|j| \leq i$. For instance, for the extreme value $j = i$, we consider the unique sequence of clopen sets $Z^{(k)} \in \calP_{m+k}$ for $k \in \{0,1,\dots,2i-1\}$ satisfying
$$Z_2 \subseteq Z^{(2i-1)}, \quad Z^{(k)} \subseteq h(Z^{(k-1)}) \text{ for $k$ odd}, \quad \text{ and } \quad Z^{(k)} \subseteq Z^{(k-1)} \text{ for $k$ even}.$$
It follows that $Z_2 \subseteq h^i(Z^{(0)})$, which implies $Z^{(0)} = Z_1$. Hence, $Z_2$ is connected to $Z_1$ via the alternating path
$$\mu = f(Z^{(1)},Z_1) \, e(Z^{(2)},Z^{(1)}) \, f(Z^{(3)},Z^{(2)}) \, \cdots \, e(Z^{(2i-2)},Z^{(2i-3)}) \, f(Z^{(2i-1)},Z^{(2i-2)}) \, e(Z_2,Z^{(2i-1)}).$$
This concludes the proof ot $(ii) \implies (i)$, and hence of the lemma.
\end{proof}

This lemma enables us to prove an elegant formula characterizing the truncated orbits of a clopen set $Z \in \calP_m$ under $h$.

\begin{proposition}\label{proposition-nice.formula}
For any $Z \in \calP_m$ and any $i \geq 0$, it holds that
$$W(Z;m,i) := \bigsqcup_{\substack{Z' \in \calP_{m+2i} \\ Z' \leadsto Z}} Z' = \bigcup_{|j| \leq i} h^j(Z).$$
\end{proposition}
\begin{proof}
This is an immediate consequence of Lemma \ref{lemma-connected.translates}. Indeed, the inclusion $W(Z;m,i) \subseteq \bigcup_{|j| \leq i} h^j(Z)$ is immediate due to the equivalence between (i) and (iii).

For the reverse inclusion, take $x \in h^j(Z)$ for some $|j| \leq i$ and consider a clopen set $Z' \in \calP_{m+2i}$ containing $x$. Then $Z' \cap h^j(Z) \neq \emptyset$, so $Z' \leadsto Z$ by the equivalence between (i) and (ii). Since $x \in Z'$, we conclude that $x \in W(Z;m,i)$, as required.
\end{proof}

We are now in a position to give a full proof of Theorem \ref{theorem:ess_min}. We will write $\calO(x)$ for the orbit of $x \in X$, namely the set $\{h^j(x) \mid j \in \Z\}$.

\begin{proof}[Proof of Theorem \ref{theorem:ess_min}]
Suppose first that $(X,h)$ is essentially minimal, with unique minimal set $Y \subseteq X$. For any $y \in Y$, it holds that $Y = \ol{\calO(y)}$.

Fix one such point $y \in Y$. For each $n \geq 0$, let $Z_n(y)$ be the unique clopen set in $\calP_n$ containing $y$. We claim that the infinite blue path $\gamma := e_0 \, e_1 \, e_2 \, \cdots$, with $e_i = e(Z_{i+1}(y),Z_i(y))$, satisfies property (EM). Suppose this is not the case; this implies in particular that there exists a vertex $Z := Z_m(y) \in \calP_m$ for which, for any $i \in \N_0$, we can find a vertex  in $\calP_{m+2i}$ which is not connected to $Z$. This means precisely that
$$W(Z;m,i) = \bigsqcup_{\substack{Z' \in \calP_{m+2i} \\ Z' \leadsto Z}} Z' \neq X.$$
Observe that $W(Z;m,i) \subseteq W(Z;m,i+1)$. Indeed, given $Z' \in \calP_{m+2i}$, we can always decompose it as
$$Z' = \bigsqcup_{\substack{Z'' \in \calP_{m+2i+2} \\ Z'' \subseteq Z'}} Z''$$
because the sequence of partitions is refined. Therefore,
$$W(Z;m,i) = \bigsqcup_{\substack{Z' \in \calP_{m+2i} \\ Z' \leadsto Z}} \bigsqcup_{\substack{Z'' \in \calP_{m+2i+2} \\ Z'' \subseteq Z'}} Z'' \subseteq \bigsqcup_{\substack{Z' \in \calP_{m+2i} \\ Z' \leadsto Z}} \bigsqcup_{\substack{Z'' \in \calP_{m+2i+2} \\ Z'' \leadsto Z'}} Z'' = \bigsqcup_{\substack{Z'' \in \calP_{m+2i+2} \\ Z'' \leadsto Z}} Z'' = W(Z;m,i+1),$$
where the inclusion holds because $Z'' \subseteq Z'$ implies, in particular, that $Z'' \leadsto Z'$ (see Lemma \ref{lemma-connected.translates}). We therefore obtain a decreasing sequence of non-empty clopen sets
$$X \setminus W(Z;m,0) \supseteq X \setminus W(Z;m,1) \supseteq X \setminus W(Z;m,2) \supseteq \cdots.$$
Since $X$ is compact, Cantor's intersection theorem guarantees the existence of a point
$$z \in \bigcap_{i \geq 0} X \setminus W(Z;m,i) = X \setminus \left( \bigcup_{i \geq 0} W(Z;m,i) \right).$$
On the one hand, by Proposition \ref{proposition-nice.formula}, $z \notin \bigcup_{j \in \Z} h^j(Z)$, which means that $\calO(z) \cap Z = \emptyset$. Because $Z$ is clopen, it follows that $\ol{\calO(z)} \cap Z = \emptyset$; in particular, $y \notin \ol{\calO(z)}$. On the other hand, by essential minimality, $\ol{\calO(z)} \supseteq Y \ni y$. This is a contradiction. Thus, property (EM) is satisfied for the infinite path $\gamma$ chosen at the beginning.\newline

Conversely, assume that $(F,D)$ is essentially minimal, so that there exists an infinite blue path $\gamma = e_0 \, e_1 \, \cdots$ satisfying property (EM). Write $Z_i := r(e_i) \in \calP_i$, which forms a decreasing sequence of clopen sets, and consider the unique point $y \in X$ belonging to the intersection $\bigcap_{i \geq 0} Z_i$. We will show that $Y := \ol{\calO(y)}$ is the unique minimal set for $(X,h)$.

For that purpose, let $Y'$ be any closed, non-empty, $h$-invariant subset of $X$. For a fixed $i \geq 0$, property (EM) ensures that there is some $n_i \in \N_0$ such that every vertex $Z \in \calP_{n_i}$ is connected to $Z_i$. Without loss of generality, we can assume that $n_i - i$ is even. Since $\calP_{n_i}$ is a partition of $X$, there is some vertex $Z' \in \calP_{n_i}$ that intersects $Y'$. By property (EM), this vertex $Z'$ is connected to $Z_i$. Thus, Lemma \ref{lemma-connected.translates} dictates that $Z' \subseteq h^j(Z_i)$ for some $j \in \Z$. Because $Z'$ intersects $Y'$, it immediately follows that $h^j(Z_i) \cap Y' \neq \emptyset$. By the $h$-invariance of $Y'$, we deduce that $Z_i \cap Y' \neq \emptyset$.

Since $Y'$ is a closed subset of the compact space $X$, the sets $Z_i \cap Y'$ are compact. Furthermore, because $\{Z_i\}_{i \geq 0}$ is a decreasing sequence, the intersections $\{Z_i \cap Y'\}_{i \geq 0}$ form a decreasing sequence of non-empty compact sets. By Cantor's intersection theorem, their infinite intersection must be non-empty. Thus,
$$\{y\} \cap Y' = \left(\bigcap_{i \geq 0} Z_i \right) \cap Y' = \bigcap_{i \geq 0} (Z_i \cap Y') \neq \emptyset,$$
which implies that $y \in Y'$. Because $Y'$ is closed and $h$-invariant, it follows that
$$Y = \ol{\calO(y)} \subseteq Y',$$
and we conclude that $Y$ is the unique minimal set for $(X,h)$, as claimed.
\end{proof}

As immediate consequences of the proof of Theorem \ref{theorem:ess_min}, we obtain the following corollaries:

\begin{corollary}\label{corollary-unique.minimal.set}
Let $X$ be a totally disconnected compact metric space and $h$ a homeomorphism on $X$. Let $(F,D)$ denote the $h$-diagram associated to $(X,h)$ with respect to an $h$-refined sequence of partitions of $X$.

Suppose that $(F,D)$ is essentially minimal, and let $\gamma = e_0 \, e_1 \, \cdots$ be any infinite blue path in $F$ satisfying property (EM). Then the unique minimal set $Y$ of $X$ is exactly $Y = \ol{\calO(y)}$, where $y \in X$ is the unique point belonging to the intersection $\bigcap_{i \geq 0} r(e_i)$.
\end{corollary}

\begin{corollary}\label{corollary-minimality}
Let $X$ be a totally disconnected compact metric space and $h$ a homeomorphism on $X$. Let $(F,D)$ denote the $h$-diagram associated to $(X,h)$ with respect to an $h$-refined sequence of partitions of $X$.

Then $(X,h)$ is minimal if and only if every infinite blue path $\gamma \in \operatorname{Path}_{\infty}(F^1)$ satisfies property (EM).
\end{corollary}
\begin{proof}
Recall that an infinite blue path uniquely determines a point in $X$ and, conversely, every point in $X$ determines a unique infinite blue path. Hence, the "only if" implication follows \textit{mutatis mutandis} by the same steps as the "only if" implication of Theorem \ref{theorem:ess_min}.

For the "if" implication, assume every infinite blue path $\gamma \in \operatorname{Path}_{\infty}(F^1)$ satisfies property (EM). By Corollary \ref{corollary-unique.minimal.set}, the unique minimal set $Y$ is then given by $Y = \ol{\calO(y)}$, where $y \in X$ is the point corresponding to \textit{any} infinite blue path. This implies $Y = \ol{\calO(x)}$ for every point $x \in X$, and we must have $Y = X$. Therefore, $(X,h)$ is minimal.
\end{proof}

\begin{example}[the odometer map on the Cantor space]
The example from Section \ref{example-odometer} is an example of a minimal system. We give a proof using Corollary \ref{corollary-minimality}. The first thing to observe is that, for any $n \geq 0$ and any $Z \in \calP_{n+1}$, the odometer map $\operatorname{ad}$ defines a permutation of the set $\calP_{n+1}$ which is transitive, meaning that for any $Z \in \calP_{n+1}$, it holds that
$$\calP_{n+1} = \{Z, \operatorname{ad}(Z), \operatorname{ad}^2(Z), \dots, \operatorname{ad}^{2^{n+1}-1}(Z)\}.$$
More generally, given any cylinder set $Z = [\ul{\epsilon_0} \, \cdots \, \epsilon_n \, \epsilon_{n+1} \, \cdots \, \epsilon_{n+r}] \in \calP_{n+r+1}$, $n,r \geq 0$, there exists an integer $k := k(Z) \in \{0,1,\dots,2^{n+1}-1\}$ such that
$$\{[\ul{\delta_0} \, \delta_1 \, \cdots \, \delta_n \, \epsilon_{n+1} \, \cdots \, \epsilon_{n+r}] \mid \delta_0,\dots,\delta_n \in \{0,1\}\} = \{\operatorname{add}^{-k}(Z), \dots, \operatorname{add}^{-1}(Z), Z, \operatorname{ad}(Z), \dots, \operatorname{ad}^{-k+2^{n+1}-1}(Z)\}.$$

Fix now $Z = [\ul{\delta_0} \, \delta_1 \, \cdots \, \delta_n] \in \calP_{n+1}$, and define $i := 2^{n+1}-1$. Take any $Z' = [\ul{\epsilon_0} \, \epsilon_1 \, \cdots \, \epsilon_n \, \epsilon_{n+1} \, \cdots \, \epsilon_{n+2i}] \in \calP_{n+1+2i}$. If we consider the auxiliary clopen set $\tilde{Z}' := [\ul{\delta_0} \, \delta_1 \, \cdots \, \delta_n \, \epsilon_{n+1} \, \cdots \, \epsilon_{n+2i}]$, then $\tilde{Z}' \subseteq Z$ and, by the previous observation, there is some $j \in \Z$ with $|j| \leq 2^{n+1}-1 = i$ such that
$$Z' = \text{ad}^j(\tilde{Z'}) \subseteq \text{ad}^j(Z),$$
so that $Z' \leadsto Z$ by Lemma \ref{lemma-connected.translates}. This shows that every vertex in $\calP_{n+1+2i}$ is connected to $Z$. Since $Z$ was an arbitrary vertex in $\calP_{n+1}$, we conclude that every infinite blue path in $F$ satisfies property (EM), and hence the system is minimal by Corollary \ref{corollary-minimality}.
\end{example}

\begin{example}[the shift map on the one-point compactification of the integers]
The example from Section \ref{example-compactified.odometer} is an example of an essentially minimal dynamical system, with unique minimal set $Y = \{\infty\}$. We show this using Theorem \ref{theorem:ess_min}.

Our candidate for an infinite blue path satisfying property (EM) is already highlighted in Figure \ref{figure-compactified.odometer}, namely $\gamma = e_0 \, e_1 \, e_2 \, \cdots$ with $e_0 = e(V_0,X)$ and $e_i = e(V_i,V_{i-1})$ for $i \geq 1$. We will show that for each $i \geq 0$ there is some $n_i \in \N_0$ such that every vertex in $\calP_{n_i}$ is connected to $V_i \in \calP_{i+1}$.

Let $n_i := 3i+3 = i+1+ 2(i+1)$. Clearly, $V_{3i+2} \in \calP_{n_i}$ is connected to $V_i$, as $V_{3i+2} \subseteq V_i$. Now take any $\{j\} \in \calP_{n_i}$ with $|j| \leq i+1$. We have
$$\{j\} = \text{add}^{j-(i+1)}(\{i+1\}) \subseteq \text{add}^{j-(i+1)}(V_i)$$
in case $j \geq 0$, with $|j-(i+1)| \leq i+1$, and
$$\{j\} = \text{add}^{j+(i+1)}(\{-i-1\}) \subseteq \text{add}^{j+(i+1)}(V_i)$$
in case $j < 0$, with $|j+(i+1)| \leq i+1$. In either case, we can apply Lemma \ref{lemma-connected.translates} to deduce that $\{j\} \leadsto V_i$. The other vertices of the form $\{j\}$ with $i+1 < |j| \leq n_i-1$ are easier, since
$$\{j\} \subseteq V_{|j|-1} \subseteq V_i,$$
so $\{j\} \leadsto V_i$ by the same lemma. The same is true for $V_{n_i-1} \subseteq V_i$.

This shows that every vertex in $\calP_{n_i}$ is connected to $V_i$, and hence property (EM) is satisfied for the infinite blue path $\gamma$. By Theorem \ref{theorem:ess_min}, we conclude that the system is essentially minimal, with its unique minimal set $Y$ being the closure of the orbit of the point $\infty$, which consists solely of the point $\infty$ itself.
\end{example}

\section*{Acknowledgments}

The author is grateful to Pere Ara and Fernando Lledó for their insightful comments, which helped to improve the manuscript's exposition.

\bibliographystyle{amsplain}

\bibliography{bibliography}

\end{document}